\documentclass[11pt,oneside,british,reqno]{amsart}
\usepackage[T1]{fontenc}
\usepackage[latin9]{inputenc}
\usepackage{geometry}
\usepackage{color}
\usepackage{babel}
\usepackage{textcomp}
\usepackage{amsbsy}
\usepackage{amstext}
\usepackage{amsthm}
\usepackage{amssymb}
\usepackage{graphicx}
\PassOptionsToPackage{normalem}{ulem}
\usepackage{ulem}
\usepackage[unicode=true,pdfusetitle,
 bookmarks=true,bookmarksnumbered=false,bookmarksopen=false,
 breaklinks=false,pdfborder={0 0 1},backref=false,colorlinks=false]
 {hyperref}

\makeatletter
\numberwithin{equation}{section}
\numberwithin{figure}{section}
\theoremstyle{plain}
\newtheorem{thm}{\protect\theoremname}[section]
\theoremstyle{definition}
\newtheorem{defn}[thm]{\protect\definitionname}
\theoremstyle{remark}
\newtheorem{rem}[thm]{\protect\remarkname}
\theoremstyle{plain}
\newtheorem{cor}[thm]{\protect\corollaryname}
\theoremstyle{plain}
\newtheorem{lem}[thm]{\protect\lemmaname}
\theoremstyle{plain}
\newtheorem{prop}[thm]{\protect\propositionname}

\usepackage{lmodern}

\makeatother

\providecommand{\corollaryname}{Corollary}
\providecommand{\definitionname}{Definition}
\providecommand{\lemmaname}{Lemma}
\providecommand{\propositionname}{Proposition}
\providecommand{\remarkname}{Remark}
\providecommand{\theoremname}{Theorem}

\begin{document}
\title{On the correlations of $n^{\alpha}$ mod 1}
\date{\today}
\author{Niclas Technau and Nadav Yesha}
\begin{abstract}
A well known result in the theory of uniform distribution modulo one
(which goes back to Fejér and Csillag) states that the fractional
parts $\{n^{\alpha}\}$ of the sequence $(n^{\alpha})_{n\ge1}$ are
uniformly distributed in the unit interval whenever $\alpha>0$ is
not an integer. For sharpening this knowledge to local statistics,
the $k$-level correlation functions of the sequence $(\{n^{\alpha}\})_{n\geq1}$
are of fundamental importance. We prove that for each $k\ge2,$ the
$k$-level correlation function $R_{k}$ is Poissonian for almost
every $\alpha>4k^{2}-4k-1$.
\end{abstract}

\address{School of Mathematical Sciences, Tel Aviv University\\
Tel Aviv 69978\\
Israel }
\email{niclast@mail.tau.ac.il}
\address{Department of Mathematics, University of Haifa\\
Haifa 3498838\\
Israel}
\email{nyesha@univ.haifa.ac.il}

\maketitle
\global\long\def\J{\mathcal{J}}%

\global\long\def\C{\mathcal{C}}%

\global\long\def\Lon{L_{1}}%

\global\long\def\Ltw{L_{2}}%

\global\long\def\Lth{L_{3}}%

\global\long\def\Lfo{L_{4}}%

\global\long\def\Lgi{L_{i}}%

\global\long\def\yon{y_{1}}%

\global\long\def\ytw{y_{2}}%

\global\long\def\yth{y_{3}}%

\global\long\def\yfo{y_{4}}%

\global\long\def\ygi{y_{i}}%

\global\long\def\Van{\mathrm{Van}}%

\global\long\def\A{A}%

\global\long\def\calS{\mathcal{S}}%

\global\long\def\la{\lambda}%

\global\long\def\y{y}%

\global\long\def\bfx{\mathbf{x}}%

\global\long\def\bfy{\mathbf{y}}%

\global\long\def\bfn{\mathbf{n}}%

\global\long\def\bfm{\mathbf{m}}%

\global\long\def\bfw{\mathbf{w}}%

\global\long\def\bfu{\mathbf{u}}%

\global\long\def\bfv{\mathbf{v}}%

\global\long\def\bfz{\mathbf{z}}%

\global\long\def\bfugi{u_{i}}%

\global\long\def\bft{\mathbf{t}}%

\global\long\def\bftau{\boldsymbol{\tau}}%

\global\long\def\bfPhi{\boldsymbol{\Phi}}%

\global\long\def\eqdef{:=}%

\global\long\def\O{O}%

\global\long\def\E{\mathcal{E}}%

\global\long\def\V{\mathcal{V}}%

\global\long\def\L{\mathcal{L}}%

\global\long\def\I{\mathcal{I}}%

\global\long\def\calBN{\mathcal{X}_{k}}%

\global\long\def\calBNp{\mathcal{N}_{k-1}^{\epsilon}}%

\section{Introduction}

A real-valued sequence $\left(\vartheta_{n}\right)_{n\geq1}$ is called
\emph{equidistributed} or \emph{uniformly distributed modulo one}
if each sub-interval $[a,b]\subseteq\left[0,1\right]$ gets its fair
share of fractional parts $\left\{ \vartheta_{n}\right\} $ in the
sense that
\[
\frac{1}{N}\#\left\{ n\leq N:\,\left\{ \vartheta_{n}\right\} \in[a,b]\right\} \underset{N\rightarrow\infty}{\longrightarrow}b-a.
\]
The notion of uniform distribution modulo one has been studied intensively
since the beginning of the twentieth century, originating in Weyl's
seminal paper \emph{Über die Gleichverteilung von Zahlen mod. Eins}
\cite{Weyl}. Notable instances of such sequences, as Weyl proved,
are $\vartheta_{n}=\alpha n^{d}$ where $d\ge1$ is an integer and
$\alpha$ is irrational.

In this paper we study another natural family of sequences whose fractional
parts are equidistributed, namely 
\begin{equation}
\vartheta_{n}=n^{\alpha}\label{eq:n^alpha}
\end{equation}
where $\alpha>0$ is \emph{non-integer}. The equidistribution modulo
one of these sequences (and more generally sequences of the form $\vartheta_{n}=\beta n^{\alpha}$
with $\beta\ne0$ and non-integer $\alpha>0)$ is a corollary of Fejér's
Theorem (see, e.g., \cite[Cor. 2.1]{KN}) in the regime $0<\alpha<1$,
which was extended to $\alpha>1$ by Csillag \cite{Csillag}.

In the last couple of decades the theory of equidistribution modulo
one acquired a new facet which has developed into a highly active
area of research: \emph{local} (or fine-scale) statistics, which measure
the behaviour of a sequence on the scale of the mean gap $1/N$. These
statistics are able to distinguish between different equidistributed
sequences, and are designed to quantify the \emph{randomness} of a
sequence; they are determined (see, e.g., \cite[Appendix A]{Kurlberg-Rudnick})
by the the \emph{$k$-level correlation functions}, which are therefore
fundamental objects in this context. We first introduce the simplest
correlation function, namely the pair correlation function.

\subsection{The pair correlation function}

The \emph{pair correlation function} $R_{2}\left(x\right)$ defined
as the limit distribution (if exists)
\begin{equation}
\lim_{N\to\infty}\frac{1}{N}\#\left\{ 1\le m\ne n\le N:\,\vartheta_{n}-\vartheta_{m}\in\frac{1}{N}I+\mathbb{Z}\right\} =\int_{I}R_{2}\left(x\right)\,\text{d}x\hspace{1em}\left(I\subseteq\mathbb{R}\right)\label{eq:R_2_def}
\end{equation}
which measures the distribution of spacings between pairs of elements
modulo one (not necessarily consecutive) on the scale of  $1/N$.
In particular, we say that the pair correlation function is Poissonian
if $R_{2}\equiv1$, which is the pair correlation function of a sequence
of independent random variables drawn uniformly in the unit interval
(Poisson process). Since Poissonian pair correlation implies equidistribution
modulo one (see \cite{Larcher Grepstad: On pair correlation and discrepancy}),
studying the pair correlation function can also be viewed as a natural
sharpening of the theory of uniform distribution modulo one.

Being the most analytically accessible local statistic, the pair correlation
of sequences modulo one has attracted considerable attention starting
with the work of Rudnick and Sarnak \cite{Rudnick Sarnak: The pair correlation function of fractional parts of polynomials},
who showed that for any $d\ge2$, the sequence $(\{\alpha n^{d}\})_{n\geq1}$
has Poissonian pair correlation for almost all $\alpha\in\mathbb{R}$.
Let us stress that often parametric families of sequences are investigated,
as results for individual sequences are rarities. Indeed, even in
the quadratic case $(\{\alpha n^{2}\})_{n\geq1}$ showing Poissonian
pair correlation even for simple choices of $\alpha$, say $\alpha=\sqrt{2}$,
is an open problem.

Rudnick and Sarnak's result is an instance of a more general metric
theory of the pair correlation of sequences of the form
\begin{equation}
\vartheta_{n}\left(\alpha\right)=\alpha a_{n}\label{eq: Kronecker subsequences}
\end{equation}
where  $(a_{n})_{n\geq1}$ is a strictly increasing sequence of positive
integers. The interest in a systematic metric theory of the pair correlation
property have recently gained momentum, following the work of Aistleitner,
Larcher and Lewko \cite{Aistleitner Larcher Lewko: Additive Energy and the Hausdorff Dimension of the Exceptional Set in Metric Pair Correlation Problems}.
A crucial observation for this development was the central role of
the additive energy $E\left(\A{}_{N}\right)$ of the truncation $\A_{N}\eqdef\left\{ a_{n}:n\leq N\right\} $,
that is 
\[
E(\A{}_{N})=\#\{(a,b,c,d)\in\A{}_{N}^{4}:\,a+b=c+d\}.
\]
With the observation $N^{2}\leq E\left(\A{}_{N}\right)\leq N^{3}$
in mind, it was proved in \cite{Aistleitner Larcher Lewko: Additive Energy and the Hausdorff Dimension of the Exceptional Set in Metric Pair Correlation Problems}
that if there is some $\epsilon>0$ such that 
\[
E\left(\A{}_{N}\right)=\O(N^{3-\epsilon}),
\]
then the fractional parts of the sequence (\ref{eq: Kronecker subsequences})
have metric Poissonian pair correlation, i.e., has Poissonian pair
correlation for almost all $\alpha\in\mathbb{R}$. The previous assumption
for identifying metric Poissonian pair correlation was slackened considerably
by Bloom and Walker \cite{Bloom-Walker}, requiring only
\[
E\left(\A{}_{N}\right)=\O(N^{3}(\log N)^{-C})
\]
with a universal constant $C>0$. For further results on the additive
energy $E\left(\A{}_{N}\right)$ and applications, see \cite{Aistleitner-Lachmann-Technau,Bloom-Chow-Gafni-Walker: Additive Energy and the Metric Poissonian,Lachmann-Technau,Walker: The Primes are not Metric Poissonian}.

There are much fewer results about the pair correlation of sequences
which are not dilated integer sequences as in (\ref{eq: Kronecker subsequences}).
Metric Poissonian pair correlation was recently established by Rudnick
and Technau \cite{Rudnick-Technau} for dilations of \emph{non-integer},
\emph{lacunary }sequences (i.e., sequences satisfying $\liminf\limits _{n\to\infty}\frac{a_{n+1}}{a_{n}}>1$).
Another family of non-integer lacunary sequences are the sequences
$\vartheta_{n}\left(\alpha\right)=\alpha^{n}$ where $\alpha>1$;
these were recently studied by Aistleitner and Baker \cite{Aistleitner-Baker}
who showed Poissonian pair correlation for almost all $\alpha>1$.
For sequences of the form (\ref{eq:n^alpha}), only the case $\alpha=1/2$
has been settled: El-Baz, Marklof and Vinogradov \cite{Elbaz-Marklof-Vinogradov}
showed that the pair correlation of the sequence $(\left\{ \sqrt{n}\right\} )_{n\geq1,\sqrt{n}\notin\mathbb{Z}}$
is Poissonian -- this is somewhat surprising in light of the non-Poissonian
nearest neighbour spacing distribution established by Elkies and McMullen
\cite{Elkies-McMullen} (see §\ref{subsec:Application} below). 

\subsection{Higher order correlation functions}

The definition of the pair correlation function naturally extends
to higher correlation functions $R_{k}\left({\bf x}\right)$ ($k\ge2)$
which detect the distribution of scaled spacings between $k$-tuples
of elements modulo one. Rather than working with boxes in $\mathbb{R}^{k-1}$,
it will be technically more convenient (and equivalent) to define
$R_{k}\left({\bf x}\right)$ via functions in $C_{c}^{\infty}(\mathbb{R}^{k-1})$,
the class of $C^{\infty}$-functions from $\mathbb{R}^{k-1}$ to $\mathbb{R}$
with compact support.  Let $\calBN=\calBN\left(N\right)$ denote the
set of distinct integer $k$-tuples $\left(x_{1},\ldots,x_{k}\right)$
satisfying $1\le x_{i}\leq N$, and for ${\bf x}\in\calBN$ denote
\[
\Delta\left(\bfx,\left(\vartheta_{n}\right)\right)\eqdef\left(\vartheta_{x_{1}}-\vartheta_{x_{2}},\ldots,\vartheta_{x_{k-1}}-\vartheta_{x_{k}}\right)\in\mathbb{R}^{k-1}.
\]

\begin{defn}
\label{def:k-point-corr}Given a compactly supported function $f:\mathbb{R}^{k-1}\to\mathbb{R}$,
we define the $k$-level correlation sum by

\begin{equation}
R_{k}(f,(\vartheta_{n}),N)\eqdef\frac{1}{N}\sum_{\bfx\in\calBN}\sum_{\bfm\in\mathbb{Z}^{k-1}}f\left(N\left(\Delta\left({\bf x},\left(\vartheta_{n}\right)\right)-\bfm\right)\right).\label{eq:R_k_def}
\end{equation}
The (limiting) $k$-level correlation function $R_{k}\left({\bf x}\right)$
is defined as the limit distribution (if exists)
\begin{equation}
\lim_{N\to\infty}R_{k}(f,(\vartheta_{n}),N)=\int_{\mathbb{R}^{k-1}}f\left({\bf x}\right)R_{k}\left({\bf x}\right)\,\text{d}{\bf x}\hspace{1em}(f\in C_{c}^{\infty}(\mathbb{R}^{k-1})).\label{eq:k-point-corr}
\end{equation}

In particular, we say that $k$-level correlation function is Poissonian
if $R_{k}\equiv1$, which is the $k$-level correlation function of
independent uniform random variables. 
\end{defn}

In contrast to the well-developed metric theory of the pair correlation
property for sequences of the shape (\ref{eq: Kronecker subsequences}),
much less is known about the triple and higher order correlation functions
whose analysis is much more involved. To the best of our knowledge,
only for lacunary sequences there are fully satisfactory results,
for which Rudnick and Zaharescu \cite{Rudnick Zaharescu} proved metric
Poissonian $k$-level correlation for any $k\geq2$.

The study of sequences of polynomial growth, even in the presence
of strong arithmetic structure, consists of only a handful of partial
results. A notable example is due to Rudnick, Sarnak and Zaharescu
\cite[Thm. 1]{Rudnick Sarnak Zaharescu: The distribution of spacings between the fractional parts of n^2},
who showed Poissonian $k$-level correlation for any $k\ge2$ for
$(\{\alpha n^{2}\})_{n\geq1}$ along special subsequences of $N$
when $\alpha$ is well approximable by rationals. Indeed, the polynomial
growth of the sequence (\ref{eq:n^alpha}) is a main challenge in
the present work. 

\subsection{Main results}

We study the correlation functions of the sequences (\ref{eq:n^alpha}).
To simplify the notation, we write $R_{k}\left(f,\alpha,N\right)$
instead of $R_{k}(f,(n^{\alpha}),N)$.
\begin{thm}
\label{thm: higher order}Let $k\geq2$. The $k$-level correlation
function of $(\{n^{\alpha}\})_{n\geq1}$ is Poissonian for almost
every $\alpha>4k^{2}-4k-1$. In particular, the pair correlation is
Poissonian for almost every $\alpha>7$.
\end{thm}

In order to prove Theorem \ref{thm: higher order}, we will take an
$L^{2}$ approach. The expected value of the $k$-level correlation
sum (when averaging over $\alpha$) is asymptotic to $\int_{\mathbb{R}^{k-1}}f\left({\bf x}\right)\,\text{d}{\bf x}$
as will be shown in Proposition \ref{prop: Expectation} (for $k=2$)
and Proposition \ref{prop:Expectation-k>2} (for $k>2$). For technical
reasons that will become apparent below, it is convenient to multiply
$\int_{\mathbb{R}^{k-1}}f\left({\bf x}\right)\,\text{d}{\bf x}$ by
the harmless combinatorial factor 
\begin{equation}
C_{k}\left(N\right)\eqdef\left(1-\frac{1}{N}\right)\cdots\left(1-\frac{k-1}{N}\right),\label{eq:C_k(N)}
\end{equation}
which is exactly the number of elements of $\calBN$ divided by $N^{k}$.
The following \emph{definition} for the variance is therefore natural.
\begin{defn}
Let $\mathcal{I}\subseteq\mathbb{R}_{>0}$ be an interval. The variance
of the $k$-level correlation sum $R_{k}\left(f,\alpha,N\right)$
with respect to $\mathcal{I}$ is defined as
\[
\mathrm{Var}\left(R_{k}\left(f,\cdot,N\right),\mathcal{I}\right)\eqdef\int_{\mathcal{\mathcal{I}}}\left(R_{k}\left(f,\alpha,N\right)-C_{k}\left(N\right)\int_{\mathbb{R}^{k-1}}f\left({\bf x}\right)\,\text{d}{\bf x}\right)^{2}\,\mathrm{d}\alpha.
\]
\end{defn}

We will deduce Theorem \ref{thm: higher order} from the following
variance bound.
\begin{thm}
\label{thm: variance bound}Let $k\geq2$, $A>4k^{2}-4k-1$ and $\J=\left[A,A+1\right]$.
There exists $\rho=\rho\left(A\right)>0$ such that 
\begin{equation}
\mathrm{Var}\left(R_{k}\left(f,\cdot,N\right),\J\right)=\O\left(N^{-\rho}\right)\label{eq: variance bound}
\end{equation}
as $N\to\infty$.
\end{thm}

\begin{rem}
Fix any $\beta\ne0$; the generalization of the above theorems to
the sequences $(\left\{ \beta n^{\alpha}\right\} )_{n\geq1}$ is straightforward.
\end{rem}

\subsection{\label{subsec:Application}Application: nearest neighbour spacing
distribution}

We may use Theorem \ref{thm: higher order} to obtain information
about various local statistics, which are determined by the $k$-level
correlation functions $R_{k}\left({\bf x}\right)$. A natural statistic
to consider is the \emph{nearest neighbour spacing distribution} (also
called the gap distribution), which is the limiting distribution $P\left(s\right)$
(if exists) of the gaps between consecutive elements (modulo one)
of the sequence scaled to have a unit mean. More precisely, if we
let
\[
\vartheta_{\left(1\right)}^{N}\le\vartheta_{\left(2\right)}^{N}\le\dots\le\vartheta_{\left(N\right)}^{N}\le\vartheta_{\left(N+1\right)}^{N},
\]
denote the first $N+1$ \emph{ordered} elements of $\left\{ \vartheta_{n}\right\} $,
then $P\left(s\right)$ is defined as the limit distribution (if exists)
\begin{equation}
\lim_{N\to\infty}g\left(x,(\vartheta_{n}),N\right)=\int_{0}^{x}P\left(s\right)\,\text{d}s\label{eq: weak convergence}
\end{equation}
where 
\[
g\left(x,(\vartheta_{n}),N\right)\eqdef\frac{1}{N}\#\left\{ n\le N:\,N\bigl(\vartheta_{\left(n+1\right)}^{N}-\vartheta_{\left(n\right)}^{N}\bigr)\le x\right\} .
\]
A strong indication for randomness of a sequence $(\{\vartheta_{n}\})_{n\ge1}$
is a Poissonian nearest neighbour distribution, that is $P\left(s\right)=e^{-s}$,
which is the nearest neighbour distribution of independent uniform
random variables.

There are only a few examples in which one can determine the gap distribution
(\ref{eq: weak convergence}). For dilations of integer lacunary sequences,
metric Poissnoian gap distribution follows from the aforementioned
metric Poissonian $k$-level correlations established in \cite{Rudnick Zaharescu}.
Another (deterministic) example is the work of Elkies and McMullen
\cite{Elkies-McMullen} on the fractional parts of the sequence $\left(\sqrt{n}\right)_{n\geq1}$.
The gap distribution turns out to be non-standard (in particular not
Poissonian) in this case, and is intimately related to the Haar measure
on the space of translates of unimodular lattices in the plane. For
the fractional parts of $\left(n^{\alpha}\right)_{n\ge1}$ with $\alpha\in\left(0,1\right)\setminus\left\{ 1/2\right\} $,
Elkies and McMullen \cite[Sec. 1]{Elkies-McMullen} conjectured that
the gap distribution is Poissonian. In fact, numerical experiments
suggest that this may hold for most, and perhaps all non-integer $\alpha\in\mathbb{R}_{>0}\setminus\left\{ 1/2\right\} $.
In this regard, while Theorem \ref{thm: higher order} does not allow
us to capture the gap distribution of $(\{n^{\alpha}\})_{n\geq1}$
completely, it ensures that for almost all large values of $\alpha$,
the distribution functions $g\left(x,(n^{\alpha}),N\right)$ can be
approximated by truncations of the Taylor series of $1-e^{-x}$ as
$N\to\infty$, so that the distribution of the gaps becomes \emph{approximately}
Poissonian.
\begin{cor}
\label{cor: approximation to gap distribution}Let $K\geq1$. For
almost all $\alpha>16K^{2}+8K-1$, we have the inequalities
\begin{align*}
\sum_{1\le k\leq2K}(-1)^{k+1}\frac{x^{k}}{k!} & \leq\liminf_{N\to\infty}g\left(x,(n^{\alpha}),N\right)\le\limsup_{N\to\infty}g\left(x,(n^{\alpha}),N\right)\leq\sum_{1\le k\leq2K-1}(-1)^{k+1}\frac{x^{k}}{k!}
\end{align*}
holding for all $x\ge0$.
\end{cor}

\subsection*{Acknowledgements}

We thank Zeév Rudnick, Jens Marklof and Daniel El-Baz for discussions
and  comments. NT received funding from the European Research Council
(ERC) under the European Union\textquoteright s Horizon 2020 research
and innovation program (Grant agreement No. 786758).\textcolor{red}{{}
}This research was, in part, carried out while NT was visiting the
International Centre for Theoretical Sciences (ICTS) in Bangalore,
whose excellent working environment are gratefully acknowledged, to
participate in the program `Smooth and Homogeneous Dynamics' (Code:
ICTS/etds2019/09).

\section{Outline of the argument}

The analysis of each correlation sum $R_{k}$ follows a general pattern.
First, let us remark that we seek to show three intermediate objectives: 
\begin{enumerate}
\item Show that the expectation of $R_{k}\left(f,\cdot,N\right)$ is asymptotic
to $\int_{\mathbb{R}^{k-1}}f\left({\bf x}\right)\,\mathrm{d}{\bf x}$.
\item The variance of $R_{k}\left(f,\cdot,N\right)$ is $\O(N^{-\rho})$
for some $\rho>0$.
\item Using the previous steps, deduce that $R_{k}\left(f,\cdot,N\right)$
converges almost surely to\\
$\int_{\mathbb{R}^{k-1}}f\left({\bf x}\right)\,\mathrm{d}{\bf x}$. 
\end{enumerate}
In other words, we set out to show that $R_{k}$ concentrates around
its mean value, which we demonstrate to be the desired limit, by taking
an $L^{2}$ approach.

As usual, the crux of the matter is to establish the variance bound.
The third step is fairly routine requiring only minor adaptations
from the standard arguments. For the sake of completeness, we decided
to detail them.

Now let us explain how we bound the variance of the pair correlation
sum
\begin{equation}
R_{2}\left(f,\alpha,N\right)=\frac{1}{N}\sum_{1\le x_{1}\ne x_{2}\le N}\sum_{m\in\mathbb{Z}}f\left(N\left(x_{1}^{\alpha}-x_{2}^{\alpha}-m\right)\right)\label{eq:PairCorrelationDef}
\end{equation}
for which the technical aspects of the analysis, which get more intricate
as $k$ increases, are still relatively simple. By using Poisson summation
and a common truncation argument, the variance can be bounded by a
sum of oscillatory integrals
\begin{equation}
\mathrm{Var}\left(R_{2}\left(f,\cdot,N\right),\J\right)\ll\frac{1}{N^{4}}\sum_{n,m}\sum_{x_{j},y_{j}}\left|\int_{\J}e(n(x_{1}^{\alpha}-x_{2}^{\alpha})-m(y_{1}^{\alpha}-y_{2}^{\alpha}))\,\mathrm{d}\alpha\right|+N^{-t},\label{eq:R2VarHeuristicBound}
\end{equation}
where the constant $t>0$ can be chosen to be arbitrarily large, and
the summation constraints are given by
\begin{equation}
\begin{array}{c}
x_{j},y_{j}\in\left[1,N\right]\quad(j=1,2),\\
n,m\in\left[-N^{1+\epsilon},N^{1+\epsilon}\right],
\end{array}\qquad\begin{array}{c}
x_{1}>x_{2},y_{1}>y_{2},\\
n\ne0,m\ne0.
\end{array}\label{eq: constraints for variance}
\end{equation}
For establishing the desired variance bound, we need to demonstrate
that the right-hand-side of (\ref{eq:R2VarHeuristicBound}) is, up
to a constant, smaller than some fixed negative power of $N$.

In order to bound the above sum, we will establish a bound for each
individual term with good dependence on the $n,m,x_{j},y_{j}$ parameters.
To this end, we use an estimate derived from a suitable modification
of van der Corput's lemma for oscillatory integrals. This provides
us with a sharp bound for the individual terms and, furthermore, with
the necessary uniformity in the parameters. For that estimate to be
applicable, we need to ensure that at any point in $\J$ at least
one of the first four\footnote{For the $k$-level correlation sum we shall consider $2k$-derivatives.}
derivatives of the phase function is large. To demonstrate this largeness
property is the crux of the matter. For verifying it, we use a ``repulsion
principle'' that quantifies how the smallness of the first three
derivatives repels the fourth derivative from being small as well
(see Figure \ref{fig:repulsion} illustrating the first four derivatives
of a phase function that we encounter).\footnote{In the case of the $k$-level correlation sum we show that at least
one of the first $2k$-derivatives is large.}\textcolor{red}{}
\begin{figure}
\centering{}\textcolor{red}{\includegraphics{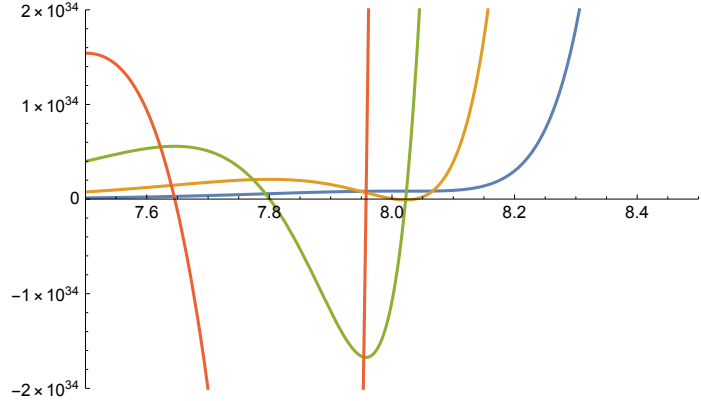}}\caption{\textcolor{black}{\label{fig:repulsion}Plot of the first four derivatives
of the phase function $\phi\left(\alpha\right)=n(x_{1}^{\alpha}-x_{2}^{\alpha})-m(y_{1}^{\alpha}-y_{2}^{\alpha}),$
where the blue curve is $\phi'$, the orange curve is $\phi''$, the
green curve is $\phi^{\left(3\right)}$ and the red curve is $\phi^{\left(4\right)}.$
Here we used the following specifications for the plot: $n=5135,m=10000$,
and $x_{1}=10000$, $x_{2}=1000$, $y_{1}=9500$, $y_{2}=7890$ in
the range $\alpha\in[7.5,8.5]$.}}
\end{figure}

\section{Preliminaries}

Before proceeding, we will introduce some notation.

\subsection{Notation}
\begin{itemize}
\item The Bachmann-Landau big $\O$ notation is used in the usual sense,
i.e., $f=\O\left(g\right)$ as $x\to\infty$ means that there exists
a constant $c>0$ such that $\vert f\left(x\right)/g\left(x\right)\vert\leq c$
holds for all $x$ sufficiently large. In order to ease the notation,
we will usually not keep track of the dependence of the implied constant
$c$ on other parameters. In particular the dependence on a (fixed)
test function $f$ shall not be explicitly mentioned.
\item We will also use the Vinogradov symbols $\ll$ (and $\gg$) in their
usual meaning in analytic number theory, that is, the statement $f\ll g$
denotes that $f=\O\left(g\right)$.
\item We will use the standard notation $e\left(z\right)=e^{2\pi iz}.$
\item We denote by $\left[k\right]\eqdef\left\{ 1,\ldots,k\right\} $ the
set of the first $k$ natural numbers. 
\item Throughout the rest of the manuscript, we denote the shifted unit
interval with left end point at $A>0$ by
\begin{equation}
\J=\J\left(A\right)\eqdef\left[A,A+1\right].\label{def: J}
\end{equation}
\end{itemize}

\subsection{Tools from harmonic analysis}

The bulk of our work is concerned with estimating one-dimensional
oscillatory integrals
\[
I\left(\phi,\J\right)\eqdef\int_{\J}e\left(\phi\left(\alpha\right)\right)\,\mathrm{d\alpha}
\]
where $\phi:\J\rightarrow\mathbb{R}$ is a $C^{\infty}$-function
(so called \emph{phase function}). The phase functions that we encounter
are of the shape 
\begin{equation}
\phi\left(\alpha\right)=\phi\left(\bfu,\bfx,\alpha\right)=\sum_{i\leq d}u_{i}x_{i}^{\alpha},\qquad\bfu=\left(u_{1},\ldots,u_{d}\right),\quad\bfx=\left(x_{1},\ldots,x_{d}\right).\label{eq: rough shape of phi}
\end{equation}
We wish to establish a bound with good dependence on the parameters
$\bfu,\bfx$ --- most importantly on the maximum norm $\left\Vert \bfx\right\Vert _{\infty}=\max_{i\leq d}\left|x_{i}\right|$.
To this end, the following well-known lemma is useful. 
\begin{lem}[Van der Corput's lemma]
\label{lem: Classic Van der Corput}Let $\phi:\J\rightarrow\mathbb{R}$
be a $C^{\infty}$-function. Fix $d\geq1$, and suppose that we have
$\bigl|\phi^{\left(d\right)}(\alpha)\bigr|\ge\lambda>0$ throughout
the interval $\J$. If $d=1$, suppose in addition that $\phi'$ is
monotone on $\J$. Then there exists a constant $C_{d}>0$ depending
only on $d$ such that
\[
\left|I\left(\phi,\J\right)\right|\le C_{d}\lambda^{-1/d}.
\]
\end{lem}

\begin{proof}
This classical bound follows from partial integration for $d=1$,
and then by induction on $d$, see Stein \cite[Ch. VIII, Prop. 2]{Stein}. 
\end{proof}
\begin{rem}
A drawback of van der Corput's lemma is that the more complicated
the phase function $\phi$ is --- bearing the shape (\ref{eq: rough shape of phi})
of $\phi$ in mind, the more difficult it is to get acceptable lower
bounds on the size of the minimum of the derivative $\phi^{\left(d\right)}$
for a given $d$. To remedy this issue, we use the following variant
of van der Corput's lemma. The key feature is that for a non-trivial
estimation of $I\left(\phi,\J\right)$, we only require that at every
point $\alpha\in\J$ at least one of the first $d$ derivatives of
$\phi$ is large --- rather than requiring that one specific derivative
is large throughout $\J$. This amounts to estimating the function
\begin{equation}
M_{d}\phi\left(\alpha\right)\eqdef\max_{1\le i\leq d}\vert\phi^{(i)}(\alpha)\vert.\label{eq:VanDef}
\end{equation}
For phrasing this variant of van der Corput's lemma, there is a small
price to pay: we need to control the number of zeros of $\phi^{\left(d\right)}$
on $\J$. 
\end{rem}

\begin{lem}
\label{lem: modified Van der Corput}Let $\phi:\J\rightarrow\mathbb{R}$
be a $C^{\infty}$-function, and let $d\geq1$. Suppose that $\phi^{\left(d\right)}$
has at most $k$ zeros, and that
\begin{equation}
M_{d}\phi\left(\alpha\right)\ge\lambda>0\label{eq:VanHypothesis}
\end{equation}
throughout the interval $\J$. If $d=1$, suppose in addition that
$\phi'$ is monotone on $\J$. Then there exists a constant $C_{d,k}>0$
depending only on $d$ and $k$ such that
\[
\left|I(\phi,\J)\right|\le C_{d,k}\lambda^{-1/d}.
\]
\end{lem}

\begin{proof}
Since $\phi^{\left(d\right)}$ has at most $k$ zeros, Rolle's theorem
implies that the number of zeros of any lower derivative $\phi^{\left(i\right)}$,
$1\le i\le d-1$, is at most $k+d-i\le k+d-1$. Hence, by splitting
the integral $I\left(\phi,\J\right)$ into $\O_{d,k}\left(1\right)$
integrals, we can assume without loss of generality that for any $1\le i\le d-1$
the function $\phi^{\left(i\right)}$ is monotone.

We will now prove by induction on $d$ that
\begin{equation}
\left|I\left(\phi,\J\right)\right|\ll_{d}\lambda^{-1/d}\label{eq: Van der Corput with crossing}
\end{equation}
where the implied constant in (\ref{eq: Van der Corput with crossing})
depends only on $d$. The case $d=1$ follows directly from Lemma
\ref{lem: Classic Van der Corput}. Assume now correctness for $d-1$
where $d\ge2$. Let $\left(a,b\right)$ be the (possibly empty) interval
of $\alpha\in\J$ satisfying
\[
M_{d-1}\phi\left(\alpha\right)=\max\limits _{1\le i\leq d-1}\vert\phi^{(i)}(\alpha)\vert<\lambda.
\]
We have
\begin{equation}
\left|I\left(\phi,\J\right)\right|\le\left|\int_{A}^{a}e\left(\phi\left(\alpha\right)\right)\,\mathrm{d\alpha}\right|+\left|\int_{a}^{b}e\left(\phi\left(\alpha\right)\right)\,\mathrm{d\alpha}\right|+\left|\int_{b}^{A+1}e\left(\phi\left(\alpha\right)\right)\,\mathrm{d\alpha}\right|.\label{eq:TriangleBound}
\end{equation}
By the assumption (\ref{eq:VanHypothesis}), the lower bound $\left|\phi^{\left(d\right)}\left(\alpha\right)\right|\ge\lambda$
holds throughout the interval $\left(a,b\right)$. Therefore Lemma
\ref{lem: Classic Van der Corput} implies that 
\begin{equation}
\left|\int_{a}^{b}e\left(\phi\left(\alpha\right)\right)\,\mathrm{d\alpha}\right|\ll_{d}\lambda^{-1/d}.\label{eq:kIsBigBound}
\end{equation}
Outside the interval $\left(a,b\right),$ we have $M_{d-1}\phi\left(\alpha\right)\ge\lambda$.
Thus, by the induction hypothesis, we infer that
\begin{align}
\left|\int_{A}^{a}e\left(\phi\left(\alpha\right)\right)\,\mathrm{d\alpha}\right|+\left|\int_{b}^{A+1}e\left(\phi\left(\alpha\right)\right)\,\mathrm{d\alpha}\right| & \ll_{d}\lambda^{-1/\left(d-1\right)}.\label{eq:kIsSmallBound}
\end{align}
Note that we may suppose that $\lambda\geq1$, since $\left|I(\phi,\J)\right|\leq1$
and, for $\lambda<1$, the desired bound plainly follows from $1<\lambda^{-1/d}$.
Now inserting the bounds (\ref{eq:kIsBigBound}) and (\ref{eq:kIsSmallBound})
into (\ref{eq:TriangleBound}) (the former dominates the latter due
to our assumption $\lambda\ge1$), gives the claimed bound (\ref{eq: Van der Corput with crossing}).
\end{proof}
The following lemma indicates how oscillatory integrals arise in our
analysis. For its proof, and later reference, we recall that for any
smooth compactly supported function $g:\mathbb{R}\rightarrow\mathbb{R}$,
the Fourier transform $\widehat{g}$ of $g$ decays rapidly in the
sense that for any arbitrarily large $t>0$ we have that
\begin{equation}
\widehat{g}\left(\xi\right)=\O\bigl(\xi^{-t}\bigr)\label{eq: fast Fourier decay}
\end{equation}
as $\left|\xi\right|\to\infty$.
\begin{lem}
\label{lem:TruncationLemma}Let $f\in C_{c}^{\infty}\left(\mathbb{R}\right)$,
and let $\epsilon>0$. Then for all $t>0$ we have
\begin{equation}
R_{2}\left(f,\alpha,N\right)=\frac{1}{N^{2}}\sum_{\left|n\right|\le N^{1+\epsilon}}\widehat{f}\left(\frac{n}{N}\right)\sum_{1\le x_{1}\ne x_{2}\le N}e\left(n\left(x_{1}^{\alpha}-x_{2}^{\alpha}\right)\right)+\O\bigl(N^{-t}\bigr)\label{eq: truncated Fourier representation}
\end{equation}
as $N\to\infty$.
\end{lem}

\begin{proof}
The Poisson summation formula applied to (\ref{eq:PairCorrelationDef})
yields the identity 
\begin{equation}
R_{2}\left(f,\alpha,N\right)=\frac{1}{N^{2}}\sum_{n\in\mathbb{Z}}\widehat{f}\left(\frac{n}{N}\right)\sum_{1\le x_{1}\ne x_{2}\le N}e\left(n\left(x_{1}^{\alpha}-x_{2}^{\alpha}\right)\right),\label{eq: Pair Corr Fourier Rep}
\end{equation}
and we want to truncate the right hand side.

Due to (\ref{eq: fast Fourier decay}) and the trivial bound 
\[
\biggl|\sum_{1\le x_{1}\ne x_{2}\le N}e\left(n\left(x_{1}^{\alpha}-x_{2}^{\alpha}\right)\right)\biggr|\leq N^{2},
\]
we have 
\begin{equation}
\sum_{\left|n\right|>N^{1+\epsilon}}\widehat{f}\left(\frac{n}{N}\right)\sum_{1\le x_{1}\ne x_{2}\le N}e\left(n\left(x_{1}^{\alpha}-x_{2}^{\alpha}\right)\right)\ll N^{2+s}\sum_{n>N^{1+\epsilon}}n^{-s}\ll N^{2+s-\left(1+\epsilon\right)\left(s-1\right)}.\label{eq: tail of Fourier representation}
\end{equation}
Taking $s$ suitably large so that
\[
s-\left(1+\epsilon\right)\left(s-1\right)=1+\epsilon-\epsilon s<-t,
\]
the right hand side of (\ref{eq: tail of Fourier representation})
is $<N^{2-t}.$ This implies (\ref{eq: truncated Fourier representation}),
concluding the proof.
\end{proof}

\section{Repulsion principles}

In order to make Lemma \ref{lem: modified Van der Corput} usable
for computing the pair and higher order correlations, we need to control
the $M$-function (\ref{eq:VanDef}) of functions as in (\ref{eq: rough shape of phi}).
In the present section, we show that irrespective\emph{ }of the choice
of $\alpha$ some derivative of such a function is large.

Recall that if 
\begin{equation}
V_{d}=\begin{bmatrix}L_{1} & L_{2} & \dots & L_{d}\\
L_{1}^{2} & L_{2}^{2} & \dots & L_{d}^{2}\\
\vdots & \vdots & \ddots & \vdots\\
L_{1}^{d} & L_{2}^{d} & \dots & L_{d}^{d}
\end{bmatrix}\label{eq:VDmatrix}
\end{equation}
is the Vandermonde matrix corresponding to distinct nonzero numbers
$L_{1},.\dots,L_{d}$, then the inverse Vandermonde matrix is given
by $V_{d}^{-1}=\left[a_{ij}\right]$, where
\begin{equation}
a_{ij}=\frac{\left(-1\right)^{j-1}\sum\limits _{\substack{1\le m_{1}<\dots<m_{d-j}\le d\\
m_{1},\dots,m_{d-j}\ne i
}
}L_{m_{1}}\cdots L_{m_{d-j}}}{L_{i}\prod\limits _{\begin{subarray}{c}
1\le m\le d\\
m\ne i
\end{subarray}}\left(L_{m}-L_{i}\right)}\label{eq:inverseVD}
\end{equation}
(see, e.g., \cite[Ex. 40]{Knuth}).

We require the following lemma.
\begin{lem}
\label{lem:VDbound}Let $d\ge2$ be an integer, let
\[
2\le x_{1}<x_{2}<\dots<x_{d}\le N
\]
be real numbers, and denote $L_{i}=\log x_{i}$ ($1\le i\le d)$.
Let $V_{d}$ be the Vandermonde matrix (\ref{eq:VDmatrix}) corresponding
to the numbers $L_{1},.\dots,L_{d}$. Let $w\in\mathbb{R}^{d}$, and
denote $y=V_{d}w$. Then for all $\epsilon>0$, there exists a constant
$C_{d,\epsilon}>0$ depending only on $d$ and $\epsilon$ such that
\[
\left\Vert y\right\Vert _{\infty}\ge C_{d,\epsilon}\left\Vert w\right\Vert _{\infty}x_{d}^{1-d}N^{-\epsilon}\prod\limits _{m=1}^{d-1}h_{m},
\]
where $h_{m}=x_{m+1}-x_{m}$ for $1\le m\le d-1$.
\end{lem}

\begin{proof}
Let $V_{d}^{-1}=\left[a_{ij}\right]$, where $a_{ij}$ is given by
(\ref{eq:inverseVD}). For every $1\le i,j\le d$ we have
\begin{equation}
a_{ij}\ll_{d,\epsilon}N^{\epsilon}\frac{1}{\prod\limits _{\begin{subarray}{c}
1\le m\le d\\
m\ne i
\end{subarray}}\left|L_{m}-L_{i}\right|}\label{eq:aij_bound}
\end{equation}
where the implied constant in (\ref{eq:aij_bound}) depends only on
$d$ and $\epsilon$.

For all $t>-1$, we have the inequality $\log\left(1+t\right)\le t$.
So, for $m=1,\dots,i-1$, we infer
\[
\left|L_{m}-L_{i}\right|=-\log\frac{x_{m}}{x_{i}}=-\log\left(1+\frac{x_{m}-x_{i}}{x_{i}}\right)\ge\frac{x_{i}-x_{m}}{x_{i}}\ge\frac{h_{m}}{x_{d}}
\]
and, for $m=i+1,\dots,d$, we have
\[
\left|L_{m}-L_{i}\right|=-\log\frac{x_{i}}{x_{m}}=-\log\left(1+\frac{x_{i}-x_{m}}{x_{m}}\right)\ge\frac{x_{m}-x_{i}}{x_{m}}\ge\frac{h_{m-1}}{x_{d}}.
\]
Hence,
\[
a_{ij}\ll_{d,\epsilon}\frac{x_{d}^{d-1}N^{\epsilon}}{\prod\limits _{m=1}^{d-1}h_{m}}.
\]
Thus, we have found a uniform bound for the elements of $V_{d}^{-1}$,
and since all matrix norms are equivalent, we conclude that
\[
\left\Vert w\right\Vert _{\infty}=\left\Vert V_{d}^{-1}y\right\Vert _{\infty}\le\left\Vert V_{d}^{-1}\right\Vert _{\infty}\left\Vert y\right\Vert _{\infty}\ll_{d,\epsilon}\frac{x_{d}^{d-1}N^{\epsilon}}{\prod\limits _{m=1}^{d-1}h_{m}}\left\Vert y\right\Vert _{\infty}.
\]
\end{proof}
We can now bound the $M$-function (\ref{eq:VanDef}) from below for
functions of the form (\ref{eq: rough shape of phi}).
\begin{lem}
\label{lem:VanBdPhi}Let $d\ge2$ be an integer, and let $u_{1},\dots,u_{d}$
be nonzero real numbers. Given real numbers $2\le x_{1}<x_{2}<\dots<x_{d}\le N$,
we define 
\[
\phi\left(\alpha\right):=\sum_{r\leq d}u_{r}x_{r}^{\alpha}\hspace{1em}\left(\alpha\in\J=\left[A,A+1\right]\right).
\]
Furthermore, let $\epsilon>0$ and define 
\[
\lambda=N^{-\epsilon}\left|u_{d}\right|x_{d}^{A+1-d}\prod\limits _{m=1}^{d-1}h_{m}.
\]
where $h_{m}=x_{m+1}-x_{m}$ ($1\le m\le d-1)$. Then there exists
a constant $C_{d,\epsilon}>0$, depending only on $d$ and $\epsilon$,
such that
\begin{equation}
M_{d}\phi\left(\alpha\right)\ge C_{d,\epsilon}\lambda>0\label{eq:VanBdPhi}
\end{equation}
throughout the interval $\J$.
\end{lem}

\begin{proof}
Denote $w=\left(u_{1}x_{1}^{\alpha},\dots,u_{d}x_{d}^{\alpha}\right)$,
and let $L_{i}=\log x_{i}$ ($1\le i\le d)$. Then
\[
\bigl(\phi^{\left(1\right)}\left(\alpha\right),\dots,\phi^{\left(d\right)}\left(\alpha\right)\bigr)^{T}=V_{d}w^{T},
\]
where $V_{d}$ is the Vandermonde matrix (\ref{eq:VDmatrix}) corresponding
to the numbers $L_{1},.\dots,L_{d}$.

By Lemma \ref{lem:VDbound}, we infer that
\[
M_{d}\phi\left(\alpha\right)\ge C_{d,\epsilon}\left\Vert w\right\Vert _{\infty}x_{d}^{1-d}N^{-\epsilon}\prod\limits _{m=1}^{d-1}h_{m}\ge C_{d,\epsilon}N^{-\epsilon}\left|u_{d}\right|x_{d}^{A+1-d}\prod\limits _{m=1}^{d-1}h_{m},
\]
where $C_{d,\epsilon}>0$ is a constant, depending only on $d$ and
$\epsilon$. This is exactly (\ref{eq:VanBdPhi}).
\end{proof}
We require the following simple bound on the number of zeros of functions
$\phi$ as in (\ref{eq: rough shape of phi}).
\begin{lem}
\label{lem:ZeroNumBound}Let $d\ge1$ be an integer, let $u_{1},\dots,u_{d}$
be nonzero real numbers, and let $x_{1},\dots,x_{d}$ be distinct
(strictly) positive numbers. Then the function
\[
\phi\left(\alpha\right)=\sum_{r\leq d}u_{r}x_{r}^{\alpha}\hspace{1em}\left(\alpha\in\mathbb{R}\right)
\]
has at most $d-1$ zeros. 
\end{lem}

\begin{proof}
The proof is by induction on $d$. For $d=1$ the correctness of the
statement is clear. Assume that the lemma is true for $d-1$ ($d\ge2)$,
and let
\[
\phi\left(\alpha\right)=\sum_{r\leq d}u_{r}x_{r}^{\alpha}.
\]
The zeros of $\phi$ are exactly the zeros of the function
\[
\tilde{\phi}\left(\alpha\right)=\sum_{r\leq d-1}\tilde{u}_{r}\tilde{x}_{r}^{\alpha}+1,
\]
where $\tilde{u}_{r}=\frac{u_{r}}{u_{d}}$, and $\tilde{x}_{r}=\frac{x_{r}}{x_{d}}$
($1\le r\le d-1$), since $\phi\left(\alpha\right)=u_{d}x_{d}^{\alpha}\tilde{\phi}\left(\alpha\right)$.
Moreover,
\[
\tilde{\phi}'\left(\alpha\right)=\sum_{r\leq d-1}v_{r}\tilde{x}_{r}^{\alpha},
\]
where $v_{r}=\tilde{u}_{r}\log\tilde{x}_{r}$ ($1\le i\le d-1$).

Clearly, the numbers $v_{1},\dots,v_{d-1}$ are nonzero and $\tilde{x}_{1},\dots,\tilde{x}_{d-1}$
are distinct. Therefore, by the induction hypothesis, $\tilde{\phi}'$
has at most $d-2$ zeros. Hence, by Rolle's theorem, $\tilde{\phi}$
has at most $d-1$ zeros, completing the proof.
\end{proof}
We are ready to prove the main lemma of this section, obtaining an
upper bound for integrals with phase functions of the form (\ref{eq: rough shape of phi}).
\begin{lem}
\label{lem:MainRepulsionLemma}Let $d\ge2$ be an integer, let $u_{1},\dots,u_{d}$
be nonzero real numbers, and let 
\[
2\le x_{1}<x_{2}<\dots<x_{d}\le N
\]
be real numbers. Denote 
\[
\phi\left(\alpha\right):=\sum_{r\leq d}u_{r}x_{r}^{\alpha}\hspace{1em}\left(\alpha\in\J=\left[A,A+1\right]\right).
\]
Then for all $\epsilon>0$, there exists a constant $C_{d,\epsilon}>0$,
depending only on $d$ and on $\epsilon$, such that
\begin{equation}
\left|I\left(\phi,\J\right)\right|\le C_{d,\epsilon}\lambda^{-1/d},\label{eq:repulsionBound}
\end{equation}
where
\begin{equation}
\lambda=N^{-\epsilon}\left|u_{d}\right|x_{d}^{A+1-d}\prod\limits _{m=1}^{d-1}h_{m},\label{eq:lambda_formula}
\end{equation}
and $h_{m}=x_{m+1}-x_{m}$ ($1\le m\le d-1)$.
\end{lem}

\begin{rem}
For $d=1$, we clearly have the (sharper) bound
\[
\left|I\left(\phi,\J\right)\right|\le C\frac{1}{\left|u_{1}\right|x_{1}^{A}\log x_{1}}
\]
where $C>0$ is an absolute constant, which follows directly from
Lemma \ref{lem: Classic Van der Corput}.
\end{rem}

\begin{proof}
For any $k\ge0,$ we have
\[
\phi^{\left(k\right)}\left(\alpha\right)=\sum_{r\leq d}v_{r}x_{r}^{\alpha}
\]
where $v_{r}=u_{r}\left(\log x_{r}\right)^{k}$. Hence, by Lemma \ref{lem:ZeroNumBound},
for any $k$ the function $\phi^{\left(k\right)}$ has at most $d-1$
zeros, and in particular this is true for $\phi^{\left(d\right)}$.

By Lemma \ref{lem:VanBdPhi}, we have
\begin{equation}
M_{d}\phi\left(x\right)\gg_{d,\epsilon}\lambda>0\label{eq:van_lower_bound}
\end{equation}
throughout the interval $\J$, where $\lambda$ is as in (\ref{eq:lambda_formula}),
and the implied constant in (\ref{eq:van_lower_bound}) depends only
on $d$ and $\epsilon$. Hence, the bound (\ref{eq:repulsionBound})
follows from Lemma \ref{lem: modified Van der Corput}.
\end{proof}

\section{The pair correlation}

The goal of this section is to prove the variance bound (\ref{eq: variance bound})
for the pair correlation sum, i.e., in the case $k=2$. This will
outline the strategy for bounding the variance in the more technically
involved case $k>2$, which will be treated in the next section.

\subsection{Computing the expectation}

First, we show that the expectation of $R_{2}\left(f,\cdot,N\right)$
is asymptotic to the average of $f$.
\begin{prop}
\label{prop: Expectation}Let $f\in C_{c}^{\infty}\left(\mathbb{R}\right)$
and let $\J$ be as in (\ref{def: J}). Then for all $\epsilon>0$, 

\begin{equation}
\int_{\mathcal{J}}R_{2}\left(f,\alpha,N\right)\,\mathrm{d}\alpha=\left(1-\frac{1}{N}\right)\int_{-\infty}^{\infty}f\left(x\right)\,\mathrm{d}x+\O\left(N^{-\min\left(2,A\right)+\epsilon}\right)\label{eq: expectation pair corr}
\end{equation}
as $N\to\infty$. 
\end{prop}

For the proof of Proposition \ref{prop: Expectation} and for later
reference, we require the subsequent lemma. 
\begin{lem}
\label{lem:OscillatingPair}If $n\neq0$ is a real number, then for
all $\epsilon>0$,
\begin{equation}
\frac{1}{N^{2}}\sum_{1\le x_{1}\ne x_{2}\le N}\left|\int_{\J}e\left(n\left(x_{1}^{\alpha}-x_{2}^{\alpha}\right)\right)\,\mathrm{d}\alpha\right|=\O_{\epsilon}\left(\frac{N^{-\min\left(2,A\right)+\epsilon}}{\left|n\right|}\right)\label{eq: oscillating integral expectation}
\end{equation}
as $N\to\infty$, where the implied constant in (\ref{eq: oscillating integral expectation})
depends only on $\epsilon$.
\end{lem}

\begin{proof}
By relabelling if needed, we can assume that the summation in (\ref{eq: oscillating integral expectation})
is over $x_{1}>x_{2}$. Consider the phase function 
\[
\phi\left(\alpha\right)=\phi\left(n,x_{1},x_{2},\alpha\right)\eqdef n\left(x_{1}^{\alpha}-x_{2}^{\alpha}\right).
\]
Note that the first derivative
\[
\phi'\left(\alpha\right)=n\left(x_{1}^{\alpha}\log x_{1}-x_{2}^{\alpha}\log x_{2}\right)
\]
 is monotone and nonzero on $\J$. Hence, Lemma \ref{lem: Classic Van der Corput}
yields 
\begin{align}
I\left(\phi,\J\right) & \ll\frac{1}{\min_{\alpha\in\J}\left|\phi'\left(\alpha\right)\right|}=\frac{1}{\left|n\right|}\frac{1}{x_{1}^{A}\log x_{1}-x_{2}^{A}\log x_{2}}\label{eq:Exp_I_upper_bound}
\end{align}
where the implied constant in (\ref{eq:Exp_I_upper_bound}) is absolute.

Let $h\eqdef x_{1}-x_{2}$. By the bound $\log\left(1+t\right)\le t$,
we have
\[
x_{1}^{A}\log x_{1}-x_{2}^{A}\log x_{2}\ge x_{1}^{A}\left(\log x_{1}-\log x_{2}\right)=-x_{1}^{A}\log\left(1-\frac{h}{x_{1}}\right)\ge x_{1}^{A-1}h.
\]
Therefore,
\begin{equation}
\sum_{1\le x_{2}<x_{1}\le N}\frac{1}{x_{1}^{A}\log x_{1}-x_{2}^{A}\log x_{2}}\le\sum_{1<x_{1}\le N}\frac{1}{x_{1}^{A-1}}\sum_{1\le h<x_{1}}\frac{1}{h}=\O_{\epsilon}\left(N^{-\min\left(0,A-2\right)+\epsilon}\right)\label{eq:IntSumBound}
\end{equation}
where the implied constant depends only on $\epsilon$. Thus, (\ref{eq: oscillating integral expectation})
follows from (\ref{eq:Exp_I_upper_bound}) and (\ref{eq:IntSumBound}).
\end{proof}
Next, we prove Proposition \ref{prop: Expectation}.
\begin{proof}[Proof of Proposition \ref{prop: Expectation}]
Recall that by Lemma \ref{lem:TruncationLemma}, for all $t>0$ we
have
\begin{align*}
R_{2}\left(f,\alpha,N\right) & =\frac{1}{N^{2}}\sum_{\left|n\right|\le N^{1+\epsilon}}\widehat{f}\left(\frac{n}{N}\right)\sum_{1\le x_{1}\ne x_{2}\le N}e\left(n\left(x_{1}^{\alpha}-x_{2}^{\alpha}\right)\right)+\O\bigl(N^{-t}\bigr)\\
 & =\left(1-\frac{1}{N}\right)\int_{-\infty}^{\infty}f\left(x\right)\,\mathrm{d}x+\frac{1}{N^{2}}\sum_{1\le\left|n\right|\le N^{1+\epsilon}}\widehat{f}\left(\frac{n}{N}\right)\sum_{1\le x_{1}\ne x_{2}\le N}e\left(n\left(x_{1}^{\alpha}-x_{2}^{\alpha}\right)\right)\\
 & +\O\bigl(N^{-t}\bigr).
\end{align*}
Integrating over $\alpha$, we have a bound ready for the summation
over $x_{1},x_{2}$ thanks to (\ref{eq: oscillating integral expectation}).
Using $\widehat{f}\ll1$, this yields
\begin{align*}
\int_{\mathcal{J}}R_{2}\left(f,\alpha,N\right)\,\mathrm{d}\alpha-\left(1-\frac{1}{N}\right)\int_{-\infty}^{\infty}f\left(x\right)\,\mathrm{d}x & \ll N^{-\min\left(2,A\right)+\epsilon/2}\sum_{1\leq\left|n\right|\leq N^{1+\epsilon}}\frac{1}{\left|n\right|}+N^{-t}.
\end{align*}
Choosing $t$ large enough will give our claim.
\end{proof}

\subsection{Proof of Theorem \ref{thm: variance bound} for \texorpdfstring{$k=2$}{k=2}}

We can now proceed with the proof of the bound (\ref{eq: variance bound})
for the pair correlation sum $R_{2}$, obtaining Theorem \ref{thm: variance bound}
in the particular case $k=2$. 
\begin{proof}[Proof of Theorem \ref{thm: variance bound} for $k=2$]
Let $\epsilon>0$. We denote by $\mathcal{S}=\mathcal{S}\left(N,\epsilon\right)$
the set of tuples $\bfz=\left(\bfn,\bfx\right)$ consisting of all
$\bfn=\left(n,-n,m,-m\right)\in\mathbb{Z}_{\ne0}^{4}$ satisfying
$\left\Vert \bfn\right\Vert _{\infty}\leq N^{1+\epsilon}$, and all
$\bfx=\left(x_{1},x_{2},y_{1},y_{2}\right)\in\mathbb{Z}_{>0}^{4}$
satisfying $\left\Vert \bfx\right\Vert _{\infty}\le N$ and
\begin{equation}
x_{1}>x_{2},\,y_{1}>y_{2},\hspace{1em}x_{1}=\left\Vert \bfx\right\Vert _{\infty}.\label{eq: ordering}
\end{equation}

From (\ref{eq: truncated Fourier representation}), the fact that
$\widehat{f}\ll1$, and relabelling, we deduce that for all $t>0$
\begin{align}
\mathrm{Var}\left(R_{2}\left(f,\cdot,N\right),\J\right) & =\int_{\mathcal{J}}\left(R_{2}\left(f,\alpha,N\right)-\left(1-\frac{1}{N}\right)\int_{-\infty}^{\infty}f\left(x\right)\,\mathrm{d}x\right)^{2}\,\mathrm{d}\alpha\label{eq:variance_basic_bound}\\
 & \ll\frac{1}{N^{4}}\sum_{\bfz\in\mathcal{S}}\left|I\left(\phi\left(\bfz,\cdot\right),\J\right)\right|+N^{-t}\nonumber 
\end{align}
with the phase function
\[
\phi\left({\bf z},\alpha\right)=n\left(x_{1}^{\alpha}-x_{2}^{\alpha}\right)-m\left(y_{1}^{\alpha}-y_{2}^{\alpha}\right).
\]

To proceed further, we split the parameter set $\calS$ into three
different regimes depending on several degeneracy conditions. Let
\begin{align*}
\mathcal{S}^{1}\eqdef & \left\{ \bfz\in\mathcal{S}:\,n=m,\,\#\left\{ x_{1},x_{2},y_{1},y_{2}\right\} <4\right\} ,\\
\mathcal{S}^{2}:= & \left\{ \bfz\in\mathcal{S}\setminus\mathcal{S}^{1}:\,x_{1}\ne y_{1}\right\} ,\\
\mathcal{S}^{3}:= & \left\{ \bfz\in\mathcal{S}\setminus\mathcal{S}^{1}:\,x_{1}=y_{1}\right\} ,
\end{align*}
so that
\[
\calS=\bigsqcup_{i\leq3}\calS^{i}.
\]
Further, we associate to each $\mathcal{S}^{i}$, $i\leq3$, the term
\[
T_{i}\eqdef\sum_{\bfz\in\mathcal{S}^{i}}\Bigl|I\left(\phi\left(\bfz,\cdot\right),\J\right)\Bigr|.
\]
Inserting the definition of $T_{i}$ into (\ref{eq:variance_basic_bound})
yields
\begin{equation}
\mathrm{Var}\left(R_{2}\left(f,\cdot,N\right),\J\right)\ll N^{-4}\sum_{r\leq3}T_{r}+N^{-t},\label{eq: variance in terms of Ti}
\end{equation}
and for verifying (\ref{eq: variance bound}) (when $k=2$) it is
enough to establish that for each $r$ we have $T_{r}\ll N^{4-\rho}$
for some $\rho>0$. We estimate the terms $T_{r}$ in order of their
index $r$.\uline{}\\
\\
\uline{Bounding \mbox{$T_{1}$}:} Let
\begin{align*}
\mathcal{S}^{1,1} & \eqdef\left\{ \bfz\in\mathcal{S}^{1}:\,\,\#\left\{ x_{1},x_{2},y_{1},y_{2}\right\} =2\right\} ,\\
\mathcal{S}^{1,2} & \eqdef\left\{ \bfz\in\mathcal{S}^{1}:\,\,\#\left\{ x_{1},x_{2},y_{1},y_{2}\right\} =3\right\} 
\end{align*}
so that
\[
T_{1}=\sum_{\bfz\in\mathcal{S}^{1,1}}\left|I\left(\phi\left(\bfz,\cdot\right),\J\right)\right|+\sum_{\bfz\in\mathcal{S}^{1,2}}\left|I\left(\phi\left(\bfz,\cdot\right),\J\right)\right|.
\]

For $\bfz\in\mathcal{S}^{1,1}$, we have $x_{1}=y_{1}$ and $x_{2}=y_{2}$,
and the phase function vanishes. Hence,
\[
\sum_{\bfz\in\mathcal{S}^{1,1}}\left|I\left(\phi\left(\bfz,\cdot\right),\J\right)\right|=\sum_{1\le\left|n\right|\le N^{1+\epsilon}}\sum_{1\le x_{2}<x_{1}\le N}1\ll N^{3+\epsilon}.
\]
For $\bfz\in\mathcal{S}^{1,2}$, we can assume without loss of generality
that $x_{1}=y_{1}$ and $x_{2}\ne y_{2}$. The phase function then
simplifies to the function
\[
\alpha\mapsto n\left(y_{2}^{\alpha}-x_{2}^{\alpha}\right).
\]
Therefore,
\begin{align*}
\sum_{\bfz\in\mathcal{S}^{1,2}}\left|I\left(\phi\left(\bfz,\cdot\right),\J\right)\right| & \ll\sum_{1\le\left|n\right|\le N^{1+\epsilon}}\sum_{1\le x_{1}\le N}\sum_{1\le x_{2}\ne y_{2}\le N}\left|I\bigl(\alpha\mapsto n\left(y_{2}^{\alpha}-x_{2}^{\alpha}\right),\J\bigr)\right|\\
 & =N\sum_{1\le\left|n\right|\le N^{1+\epsilon}}\sum_{1\le x_{2}\ne y_{2}\le N}\left|I\bigl(\alpha\mapsto n\left(y_{2}^{\alpha}-x_{2}^{\alpha}\right),\J\bigr)\right|.
\end{align*}
We apply Lemma \ref{lem:OscillatingPair} to deduce that 
\[
\sum_{1\le\left|n\right|\le N^{1+\epsilon}}\sum_{1\le x_{2}\ne y_{2}\le N}\left|I\bigl(\alpha\mapsto n\left(y_{2}^{\alpha}-x_{2}^{\alpha}\right),\J\bigr)\right|\ll N^{2-\min\left(2,A\right)+\epsilon}.
\]
Hence, 
\begin{equation}
T_{1}\ll N^{3+\epsilon}+N^{3-\min\left(2,A\right)+\epsilon}\ll N^{3+\epsilon}.\label{eq: bound on T1}
\end{equation}
\uline{}\\
\uline{Bounding \mbox{$T_{2}$}:} For $2\le d\le4$, let
\begin{align*}
\mathcal{S}^{2,d} & \eqdef\left\{ \bfz\in\mathcal{S}^{2}:\,\#\left(\left\{ x_{1},x_{2},y_{1},y_{2}\right\} \setminus\left\{ 1\right\} \right)=d\right\} ,
\end{align*}
so that
\begin{equation}
T_{2}=\sum_{2\leq d\leq4}\sum_{\bfz\in\mathcal{S}^{2,d}}\left|I\left(\phi\left(\bfz,\cdot\right),\J\right)\right|.\label{eq:T_2_decomp}
\end{equation}

For ${\bf z}\in\mathcal{S}^{2,d}$, the phase function $\phi$ consists
of $d$ non-constant terms with non-vanishing coefficients. Since
$x_{1}\ne y_{1},$ the leading coefficient of $x_{1}^{\alpha}$ is
$n$. Invoking Lemma \ref{lem:MainRepulsionLemma} (recall that $x_{1}=\left\Vert \bfx\right\Vert _{\infty})$,
we obtain
\begin{align*}
\sum_{\bfz\in\mathcal{S}^{2,d}}\left|I\left(\phi\left(\bfz,\cdot\right),\J\right)\right| & \ll N^{\epsilon/d}\sum_{1\le\left|n\right|,\left|m\right|\le N^{1+\epsilon}}\left|n\right|^{-\frac{1}{d}}\sum_{x_{1}\leq N}x_{1}^{1-\frac{A}{d}-\frac{1}{d}}\sum_{h_{1},\dots,h_{d-1}\leq x_{1}}h_{1}^{-1/d}\cdots h_{d-1}^{-1/d}\\
 & \ll N^{1+\epsilon+\epsilon/d}\sum_{1\le\left|n\right|\le N^{1+\epsilon}}\left|n\right|^{-\frac{1}{d}}\sum_{x_{1}\leq N}x_{1}^{1-\frac{A}{d}-\frac{1}{d}}\sum_{h_{1},\dots,h_{d-1}\leq x_{1}}h_{1}^{-1/d}\cdots h_{d-1}^{-1/d}.
\end{align*}
The innermost summation over the $h_{i}\leq x_{1}$ variables equals
\[
\biggl(\sum_{h\leq x_{1}}h^{-\frac{1}{d}}\biggr)^{d-1}\ll x_{1}^{d-2+\frac{1}{d}}.
\]
Since the sum over $n$ is $\ll N^{\left(1-\frac{1}{d}\right)\left(1+\epsilon\right)}$,
we deduce that
\begin{equation}
\sum_{\bfz\in\mathcal{S}^{2,d}}\left|I\left(\phi\left(\bfz,\cdot\right),\J\right)\right|\ll N^{2-1/d+2\epsilon}\sum_{x_{1}\leq N}x_{1}^{d-1-\frac{A}{d}}\ll N^{2-1/d-\min\left(0,\frac{A}{d}-d\right)+3\epsilon}.\label{eq:T_2_main_bound}
\end{equation}
Substituting (\ref{eq:T_2_main_bound}) back into (\ref{eq:T_2_decomp}),
we obtain
\begin{equation}
T_{2}\ll\sum_{2\leq d\leq4}N^{2-1/d-\min\left(0,\frac{A}{d}-d\right)+3\epsilon}\ll N^{4-\rho}\label{eq: bound on T2}
\end{equation}
for some $\rho>0$ as long as we have $-2-\frac{1}{d}-\frac{A}{d}+d<0\iff A>d^{2}-2d-1$
for all $2\le d\le4$ , which is equivalent to the condition $A>7$.
\uline{}\\
\uline{}\\
\uline{Bounding \mbox{$T_{3}$}:} For $1\le d\le3$, let
\begin{align*}
\mathcal{S}^{3,d} & \eqdef\left\{ \bfz\in\mathcal{S}^{3}:\,\#\left(\left\{ x_{1},x_{2},y_{2}\right\} \setminus\left\{ 1\right\} \right)=d\right\} ,
\end{align*}
so that
\[
T_{3}=\sum_{d\leq3}\sum_{\bfz\in\mathcal{S}^{3,d}}\left|I\left(\phi\left(\bfz,\cdot\right),\J\right)\right|.
\]

For ${\bf z}\in\mathcal{S}^{3,d}$, the phase function $\phi$ consists
of $d$ non-constant terms with non-vanishing coefficients. Now we
have $x_{1}=y_{1},$ and therefore the leading coefficient of $x_{1}^{\alpha}$
is $l:=n-m\ne0$. Hence, Lemma \ref{lem:MainRepulsionLemma} yields
\begin{align*}
\sum_{\bfz\in\mathcal{S}^{3,d}}\left|I\left(\phi\left(\bfz,\cdot\right),\J\right)\right| & \ll N^{\epsilon/d}\sum_{\substack{1\le\left|m\right|\le N^{1+\epsilon}\\
1\le\left|l\right|\le2N^{1+\epsilon}
}
}\left|l\right|^{-1/d}\sum_{x_{1}\leq N}x_{1}^{1-\frac{A}{d}-\frac{1}{d}}\sum_{h_{1},\dots,h_{d-1}\leq x_{1}}h_{1}^{-1/d}\cdots h_{d-1}^{-1/d}\\
 & \ll N^{2-1/d-\min\left(0,\frac{A}{d}-d\right)+3\epsilon}.
\end{align*}
Therefore,
\begin{equation}
T_{3}\ll\sum_{d\leq3}N^{2-1/d-\min\left(0,\frac{A}{d}-d\right)+3\epsilon}\ll N^{4-\rho}\label{eq: bound on T3}
\end{equation}
for some $\rho>0$ as long as $A>d^{2}-2d-1$ for all $1\le d\le3$,
which is equivalent to the condition $A>2$.\\
\\

To summarize, if $A>7$, then inserting into (\ref{eq: variance in terms of Ti})
the estimates of the $T_{i}$ from (\ref{eq: bound on T1}), (\ref{eq: bound on T2}),
and (\ref{eq: bound on T3}), we find that 
\begin{align*}
\mathrm{Var}\left(R_{2}\left(f,\cdot,N\right),\J\right) & \ll N^{-\rho}
\end{align*}
for some $\rho>0$.
\end{proof}

\section{Higher order correlations}

\subsection{Expectation and variance in terms of oscillatory integrals}

For $k\ge2$ and $\epsilon>0$, let $\calBNp=\calBNp\left(N\right)$
denote the set of integer $\left(k-1\right)$-tuples $\left(n_{1},\ldots,n_{k-1}\right)$
satisfying $\left|n_{i}\right|\leq N^{1+\epsilon}$.

Recall that we denoted by $\calBN$ the set of \emph{distinct} integer
$k$-tuples $\left(x_{1},\ldots,x_{k}\right)$ satisfying $1\le x_{i}\leq N$.
For ${\bf x}=\left(x_{1},\dots,x_{k}\right)\in\calBN$, denote
\[
\Delta\left(\bfx,\alpha\right)\eqdef\left(x_{1}^{\alpha}-x_{2}^{\alpha},x_{2}^{\alpha}-x_{3}^{\alpha},\dots,x_{k-1}^{\alpha}-x_{k}^{\alpha}\right),
\]
so that
\[
R_{k}\left(f,\alpha,N\right)=\frac{1}{N}\sum_{\bfm\in\mathbb{Z}^{k-1}}\sum_{\bfx\in\calBN}f\left(N\left(\Delta\left(\bfx,\alpha\right)-\bfm\right)\right).
\]
The following lemma generalizes Lemma \ref{lem:TruncationLemma}. 
\begin{lem}[Truncated Poisson summation]
\label{lem: truncated Poisson}Let $k\ge2,$ $f\in C_{c}^{\infty}(\mathbb{R}^{k-1})$,
and $\epsilon>0$. Then for all $t>0$ we have that
\begin{equation}
R_{k}\left(f,\alpha,N\right)=\frac{1}{N^{k}}\sum_{\bfn\in\calBNp}\sum_{\bfx\in\calBN}\widehat{f}\left(\frac{\bfn}{N}\right)e\left(\left\langle \Delta\left(\bfx,\alpha\right),\bfn\right\rangle \right)+\O(N^{-t})\label{eq: truncated Poisson form}
\end{equation}
as $N\to\infty$.
\end{lem}

\begin{proof}
By Poisson summation,
\[
R_{k}\left(f,\alpha,N\right)=\frac{1}{N^{k}}\sum_{\bfn\in\mathbb{Z}^{k-1}}\sum_{\bfx\in\calBN}\widehat{f}\left(\frac{\bfn}{N}\right)e\left(\left\langle \Delta\left(\bfx,\alpha\right),\bfn\right\rangle \right).
\]
Bounding the summation over ${\bf x}\in\calBN$ trivially yields
\[
\sum_{\substack{\bfn\in\mathbb{Z}^{k-1}\\
\left\Vert \bfn\right\Vert _{\infty}>N^{1+\epsilon}
}
}\sum_{\bfx\in\calBN}\widehat{f}\left(\frac{\bfn}{N}\right)e\left(\left\langle \Delta\left(\bfx,\alpha\right),\bfn\right\rangle \right)\le N^{k}\sum_{\substack{\bfn\in\mathbb{Z}^{k-1}\\
\left\Vert \bfn\right\Vert _{\infty}>N^{1+\epsilon}
}
}\left|\widehat{f}\left(\frac{\bfn}{N}\right)\right|.
\]
By the rapid decay of $\widehat{f}$, for ${\bf u}=\left(u_{1},\dots,u_{k-1}\right)\in\mathbb{R}^{k-1}$
we have
\begin{equation}
\widehat{f}\left(\mathbf{u}\right)=\O\left(\frac{1}{\left(1+\left|u_{1}\right|\right)^{s_{1}}\cdots\left(1+\left|u_{k-1}\right|\right)^{s_{k-1}}}\right)\label{eq: fast decay of multivariate Fourier transform}
\end{equation}
for any $s_{1},\dots,s_{k-1}>0$. In particular, for any $s>0$ we
have
\begin{align}
\sum_{\substack{\bfn\in\mathbb{Z}^{k-1}\\
\left\Vert \bfn\right\Vert _{\infty}>N^{1+\epsilon}
}
}\left|\widehat{f}\left(\frac{\bfn}{N}\right)\right| & \ll N^{s+2\left(k-2\right)}\sum_{\substack{\bfn\in\mathbb{Z}^{k-1}\\
n_{1}>N^{1+\epsilon}
}
}n_{1}^{-s}(1+\left|n_{2}\right|)^{-2}\cdots(1+\left|n_{k-1}\right|)^{-2}\label{eq:higher_corr_tail}\\
 & \ll N^{2k-4+s}\sum_{n>N^{1+\epsilon}}n^{-s}\ll N^{2k-4+s-\left(1+\epsilon\right)\left(s-1\right)}.\nonumber 
\end{align}
Taking $s$ suitably large so that
\[
2k-4+s-\left(1+\epsilon\right)\left(s-1\right)=2k-3+\epsilon-\epsilon s<-t,
\]
we get that the right-hand-side of (\ref{eq:higher_corr_tail}) is
$<N^{-t}$ which gives our claim.
\end{proof}
Given ${\bf n}\in\mathbb{Z}^{k-1}$, we define the vector ${\bf u}\left({\bf n}\right)=\left(u_{1}\left({\bf n}\right),\dots,u_{k}\left({\bf n}\right)\right)\in\mathbb{Z}^{k}$
by the rule
\[
u_{i}\left(\bfn\right)\eqdef\begin{cases}
n_{1}, & \mathrm{if}\,i=1,\\
n_{i}-n_{i-1}, & \mathrm{if}\,2\leq i\leq k-1\\
-n_{k-1}, & \mathrm{if}\,i=k.
\end{cases},
\]
Note that the linear map ${\bf n}\mapsto{\bf u}\left({\bf n}\right)$
is injective. Moreover, it satisfies the bound
\begin{equation}
\left\Vert {\bf u}\left({\bf n}\right)\right\Vert _{\infty}\le2\left\Vert {\bf n}\right\Vert _{\infty}\label{eq:u_n_range}
\end{equation}
and the relation
\begin{equation}
\sum_{i=1}^{k}u_{i}\left(\bfn\right)=0.\label{eq:u_n_relation}
\end{equation}

Let 
\[
\mathcal{U}_{k}^{\epsilon}=\mathcal{U}_{k}^{\epsilon}\left(N\right)=\left\{ {\bf u}=\left(u_{1},\dots,u_{k}\right)\in\mathbb{Z}^{k}:\,1\le\left\Vert \bfu\right\Vert _{\infty}\le2N^{1+\epsilon},\,u_{1}+\dots+u_{k}=0\right\} ,
\]
and note that the relations (\ref{eq:u_n_range}), (\ref{eq:u_n_relation})
imply that ${\bf u}\left({\bf n}\right)\in\mathcal{U}_{k}^{\epsilon}$
whenever $\mathbf{0}_{k-1}\ne{\bf n}\in\calBNp$.

\subsection{Degenerate regimes}

Let $K>0$ (we will take below either $K=k$ or $K=2k$), let $\bfu=\left(u_{1},\ldots,u_{K}\right)\in\mathbb{Z}^{K}$,
$\bfx=\left(x_{1},\ldots,x_{K}\right)\in\mathbb{Z}_{>0}^{K}$, and
let $\phi\left(\alpha\right)$ be a function of the form

\begin{equation}
\phi\left(\alpha\right)=\phi\left({\bf z},\alpha\right)=\sum_{i\leq K}u_{i}x_{i}^{\alpha},\quad{\bf z=\left({\bf u},{\bf x}\right)}.\label{eq:phi_def_K}
\end{equation}
For utilizing the repulsion principle, we require the derivative of
the phase function $\phi$ to genuinely depend on all the $x_{i}$
variables. While this is true throughout most of the regime, there
are certain constellations of the parameters where this basic property
fails. To illustrate this phenomenon, let us consider the oscillatory
integrals that we have already encountered when analysing the variance
of the pair correlation sum
\begin{equation}
\int_{\J}e(n(x_{1}^{\alpha}-x_{2}^{\alpha})-m(y_{1}^{\alpha}-y_{2}^{\alpha}))\,\mathrm{d}\alpha,\quad\begin{array}{c}
1\le x_{2}<x_{1}\le N,\\
1\le y_{2}<y_{1}\leq N,
\end{array}\quad1\leq\left|n\right|,\left|m\right|\leq N^{1+\epsilon}.\label{eq: oscillating integrals for the variance}
\end{equation}

Already here (different kinds of) degeneracy issues arose, yet the
combinatorics was still simple. Note that this integral can degenerate
in essentially three different ways:
\begin{enumerate}
\item Some of the variables $x_{i},y_{i}$ can be equal to $1$, e.g., $x_{2}=1$
(in fact there can be at most two such variables).
\item Some of the variables $x_{i},y_{i}$ could be identical, e.g., we
may have that $x_{1}=y_{1}$; in fact, there can be at most two pairs
of identical variables in (\ref{eq: oscillating integrals for the variance}).
\item The variables $n,m$ can be chosen in such a manner that the coefficients
in front of some terms vanish. For instance, we may have $n=m$ and
$x_{1}=y_{1}$. Moreover, a particular scenario is that the variables
are arranged in such a way that the phase function $\phi$ vanishes
identically\footnote{It turns out $\phi$ can vanish identically only in the kind of integrals
like (\ref{eq: oscillating integrals for the variance}) appearing
in the variance bounds, but not in the kind of integrals involved
in the expectation.}, when
\[
x_{1}=y_{1},\,x_{2}=y_{2},\,n=m.
\]
Fortunately, this is the only configuration for this scenario to happen,
and there are only $\O\left(N^{3+\epsilon}\right)$ such parameters.
Since the variance estimate is equipped with a normalization factor
of $N^{-4}$ the contribution from this regime is negligible.
\end{enumerate}
Each of these possible degeneracies will also occur when dealing with
the expectation and the variance of higher correlation sums, and will
need to be accounted for. 

Given $\phi$ of the form (\ref{eq:phi_def_K}), we define a measurement
of how many variables $x_{i}$ genuinely occur in the derivative $\phi'\left(\alpha\right)$.
We can clearly write
\begin{equation}
\phi'\left(\alpha\right)=\sum_{i\leq d}w_{i}z_{i}^{\alpha}\label{eq:non-deg-rep}
\end{equation}
where $0\le d\le K$, $\left\{ z_{1},\ldots,z_{d}\right\} \subseteq\left\{ x_{1},\dots,x_{K}\right\} $,
$z_{1},\dots,z_{d}\ge2$ are distinct, and $w_{1},\dots,w_{d}\ne0$.
Moreover, by the independence of the functions $\alpha\mapsto z_{i}^{\alpha}$,
the representation (\ref{eq:non-deg-rep}) is unique, and we say that
$\phi$ is $(K-d)$-degenerate. Instead of $0$-degenerate (corresponding
to $d=K$) we say that $\phi$ is non-degenerate. 

Let $\E_{k,d}^{\epsilon}=\E_{k,d}^{\epsilon}\left(N\right)$ (resp.
$\V_{k,d}^{\epsilon}=\V_{k,d}^{\epsilon}\left(N\right)$) denote the
set of all $\bfz=\left({\bf u},{\bf x}\right)\in\mathcal{U}_{k}^{\epsilon}\times\calBN$
(resp. $\bfz=\left({\bf u},{\bf v},{\bf x},{\bf y}\right)\in\left(\mathcal{U}_{k}^{\epsilon}\right)^{2}\times\calBN^{2}$)
such that $\phi\left(\bfz,\alpha\right)$ is $\left(k-d\right)$-degenerate
(resp. $\left(2k-d\right)$-degenerate). Our main goal will be to
bound the sums
\[
S(\E_{k,d}^{\epsilon},\J)\eqdef\sum_{\bfz\in\E_{k,d}^{\epsilon}}\left|I\left(\phi\left(\bfz,\cdot\right),\J\right)\right|,\qquad S(\V_{k,d}^{\epsilon},\J)\eqdef\sum_{\bfz\in\mathcal{V}_{k,d}^{\epsilon}}\left|I\left(\phi\left(\bfz,\cdot\right),\J\right)\right|.
\]
 As we shall show, these quantities control the expectation and variance
of $R_{k}$.
\begin{lem}
\label{lem: expectation for higher corr via oscil inte}Let $k\ge2,$
$f\in C_{c}^{\infty}(\mathbb{R}^{k-1})$, $\epsilon>0$, $C_{k}\left(N\right)$
be as in (\ref{eq:C_k(N)}), and 

\[
E_{k}\left(N,\J,\epsilon\right)\eqdef\frac{1}{N^{k}}\sum_{0<d\leq k}S(\E_{k,d}^{\epsilon},\J).
\]
Then for all $t>0$, as $N\to\infty$, we have that
\begin{equation}
\int_{\mathcal{J}}R_{k}\left(f,\alpha,N\right)\,\mathrm{d}\alpha-C_{k}\left(N\right)\int_{\mathbb{R}^{k-1}}f\left(\bfx\right)\,\mathrm{d}\bfx\ll E_{k}\left(N,\J,\epsilon\right)+N^{-t}.\label{eq: rep of expectation remainder}
\end{equation}
\end{lem}

\begin{proof}
By Lemma \ref{lem: truncated Poisson},
\[
\int_{\mathcal{J}}R_{k}\left(f,\alpha,N\right)\,\text{d}\alpha=\frac{1}{N^{k}}\sum_{\substack{\bfx\in\calBN\\
\bfn\in\calBNp
}
}\widehat{f}\left(\frac{\bfn}{N}\right)\int_{\mathcal{J}}e\left(\left\langle \Delta\left(\bfx,\alpha\right),\bfn\right\rangle \right)\,\text{d}\alpha+\O(N^{-t}).
\]
We observe that for $\bfn=\mathbf{0}_{k-1}$, the number of corresponding
$\bfx\in\calBN$ to choose from is
\[
\#\calBN=N\cdot\left(N-1\right)\ldots\left(N-k+1\right).
\]
Hence 
\begin{align}
\int_{\mathcal{J}}R_{k}\left(f,\alpha,N\right)\,\text{d}\alpha & -C_{k}\left(N\right)\int_{\mathbb{R}^{k-1}}f\left(\bfx\right)\,\text{d}\bfx\nonumber \\
 & =\frac{1}{N^{k}}\sum_{\substack{\bfx\in\calBN\\
{\bf 0}_{k-1}\ne\bfn\in\calBNp
}
}\widehat{f}\left(\frac{\bfn}{N}\right)\int_{\mathcal{J}}e\left(\left\langle \Delta\left(\bfx,\alpha\right),\bfn\right\rangle \right)\,\text{d}\alpha+\O(N^{-t}).
\end{align}
Clearly,
\begin{align*}
\left\langle \Delta\left(\bfx,\alpha\right),\bfn\right\rangle  & =\sum_{1\le i\leq k-1}n_{i}\left(x_{i}^{\alpha}-x_{i+1}^{\alpha}\right)=n_{1}x_{1}^{\alpha}-n_{k-1}x_{k}^{\alpha}+\sum_{2\leq i\leq k-1}\left(n_{i}-n_{i-1}\right)x_{i}^{\alpha}\\
 & =\phi\left({\bf u}\left(\bfn\right),\bfx,\alpha\right).
\end{align*}
Thus, 
\[
\int_{\mathcal{J}}e\left(\left\langle \Delta\left(\bfx,\alpha\right),\bfn\right\rangle \right)\,\text{d}\alpha=I\left(\phi\left({\bf u}\left(\bfn\right),\bfx,\cdot\right),\J\right).
\]
Summing over the different $\left(k-d\right)$-degenerate regimes
(and noting that $k$-degeneracy, corresponding to $d=0$, cannot
occur for $\bfn\ne\mathbf{0}_{k-1}$) implies
\[
\sum_{\substack{\bfx\in\calBN\\
{\bf 0}_{k-1}\ne\bfn\in\calBNp
}
}\widehat{f}\left(\frac{\bfn}{N}\right)\int_{\mathcal{J}}e\left(\left\langle \Delta\left(\bfx,\alpha\right),\bfn\right\rangle \right)\,\text{d}\alpha\ll\sum_{0<d\leq k}\sum_{\substack{\bfx\in\calBN\\
{\bf 0}_{k-1}\ne\bfn\in\calBNp\\
\bfz=\left(\bfu\left({\bf n}\right),{\bf x}\right)\in\E_{k,d}^{\epsilon}
}
}\left|I\left(\phi\left(\bfz,\cdot\right),\J\right)\right|.
\]
Finally, by the injectivity of the map ${\bf n}\mapsto{\bf u}\left({\bf n}\right)$
we have
\[
\sum_{\substack{\bfx\in\calBN\\
{\bf 0}_{k-1}\ne\bfn\in\calBNp\\
\bfz=\left(\bfu\left({\bf n}\right),{\bf x}\right)\in\E_{k,d}^{\epsilon}
}
}\left|I\left(\phi\left(\bfz,\cdot\right),\J\right)\right|\ll S(\E_{k,d}^{\epsilon},\J)
\]
which implies (\ref{eq: rep of expectation remainder}).
\end{proof}
The derivation of a bound for the variance of $R_{k}$ in terms of
oscillatory integrals is similar:
\begin{lem}
\label{lem: variance for higher corr via oscil inte} Let $k\ge2,$
$f\in C_{c}^{\infty}(\mathbb{R}^{k-1})$, and $\epsilon>0$. Then,
for all $t>0$, we have that
\begin{align}
\mathrm{Var}\left(R_{k}\left(f,\cdot,N\right),\J\right) & \ll V_{k}\left(N,\J,\epsilon\right)+N^{-t}\label{eq: rep of variance remainder}
\end{align}
as $N\to\infty$, where the term $V_{k}\left(f,N,\J\right)$ is the
given by the sum
\[
V_{k}\left(N,\J,\epsilon\right)\eqdef\frac{1}{N^{2k}}\sum_{0\leq d\leq2k}S(\V_{k,d}^{\epsilon},\J).
\]
\end{lem}

\begin{proof}
By Lemma \ref{lem: truncated Poisson}, for all $s>0$ we have
\[
\mathrm{Var}\left(R_{k}\left(f,\cdot,N\right),\J\right)=\int_{\J}\Bigl(N^{-k}\sum_{\substack{\bfx\in\calBN\\
{\bf 0}_{k-1}\ne\bfn\in\calBNp
}
}\widehat{f}\left(\frac{\bfn}{N}\right)e\left(\left\langle \Delta\left(\bfx,\alpha\right),\bfn\right\rangle \right)+\O(N^{-s})\Bigr)^{2}\:\mathrm{d}\alpha.
\]
Expanding the square and taking $s$ sufficiently large, we get that
for all $t>0$
\begin{align*}
\mathrm{Var}\left(R_{k}\left(f,\cdot,N\right),\J\right) & =N^{-2k}\sum_{\substack{\bfx,{\bf y}\in\calBN\\
{\bf 0}_{k-1}\ne\bfn,{\bf m}\in\calBNp
}
}\widehat{f}\left(\frac{\bfn}{N}\right)\widehat{f}\left(\frac{{\bf m}}{N}\right)\\
 & \times\int_{\J}e\left(\left\langle \Delta\left(\bfx,\alpha\right),\bfn\right\rangle +\left\langle \Delta\left({\bf y},\alpha\right),{\bf m}\right\rangle \right)\:\mathrm{d}\alpha+O(N^{-t}).
\end{align*}
Repeating the arguments of Lemma \ref{lem: expectation for higher corr via oscil inte}
yields the claim. 
\end{proof}

\subsection{Proof of Theorem \ref{thm: variance bound} -- the general case}

As outlined above, we wish to relate the sums $S(\E_{k,d}^{\epsilon},\J)$
and $S(\mathcal{V}_{k,d}^{\epsilon},\J)$ to quantities that only
involve those $\bfz$ for which $\phi\left(\bfz,\alpha\right)$ is
non-degenerate.

Fix $d,j\ge1$, and for each $1\le i\le d$, let $L_{i}:\mathbb{Z}^{j}\to\mathbb{Z}$
denote a nonzero linear map. Let $\mathbf{L}=\left(L_{1},\dots,L_{d}\right)$,
and consider the sums
\begin{equation}
Q\left(d,j,\mathcal{\mathbf{L}},N,\J,\epsilon\right)\eqdef\sum_{\substack{\bfm\in\mathbb{Z}^{j}\\
1\le\left\Vert \bfm\right\Vert _{\infty}\le2N^{1+\epsilon}\\
L_{i}\left({\bf m}\right)\ne0\,\left(1\le i\le d\right)
}
}\,\,\sum_{\substack{\bft=\left(t_{1},\dots,t_{d}\right)\in\mathbb{Z}^{d}\\
2\le t_{i}\le N\,\text{distinct}
}
}\left|I\left(\phi\left(\mathcal{\mathbf{L}}\left({\bf m}\right),\bft,\cdot\right),\J\right)\right|.\label{eq: Q1}
\end{equation}
Here $\phi$ is as in (\ref{eq:phi_def_K}) with $K=d$.

Surely bounding $S(\E_{k,d}^{\epsilon},\J),\,S(\mathcal{V}_{k,d}^{\epsilon},\J)$
in terms of $Q$ requires weight factors accounting for the number
of variables that $\left|I\left(\phi\left(\bfz,\cdot\right),\J\right)\right|$
does not effectively depend upon --- see the proofs of Proposition
\ref{prop: reduction expectation} and Proposition \ref{prop: variance reduction}
below. But first, we deduce an appropriate bound for $Q$.
\begin{lem}
\label{lem:Q_Bound}Fix $d,j\geq1$. For $1\leq i\leq d$, let $L_{i}:\mathbb{Z}^{j}\to\mathbb{Z}$
denote a nonzero linear map. Further, let $\mathbf{L}=\left(L_{1},\dots,L_{d}\right)$.
For every $\epsilon>0$, we have the bound 
\begin{equation}
Q\left(d,j,\mathcal{\mathbf{L}},N,\J,\epsilon\right)=\O\bigl(N^{j-\frac{1}{d}-\min\left(\frac{A}{d}-d,0\right)+\epsilon}\bigr)\label{eq: bound for cut-off}
\end{equation}
as $N\to\infty$.
\end{lem}

\begin{proof}
Let $\bfm,\bft$ be arbitrary elements in the summation of $Q$, and
assume without loss of generality that $t_{1}<t_{2}<\dots<t_{d}$.
By Lemma \ref{lem:MainRepulsionLemma},
\[
\left|I\left(\phi\left(\mathcal{\mathbf{L}}\left({\bf m}\right),\bft,\cdot\right),\J\right)\right|\ll_{d,\epsilon}\left|L_{d}\left(\bfm\right)\right|^{-\frac{1}{d}}t_{d}^{-\frac{A}{d}+1-\frac{1}{d}}\left(h_{1}\ldots h_{d-1}\right)^{-\frac{1}{d}}N^{\epsilon/d}.
\]
where $h_{i}=t_{i+1}-t_{i}$ for $i=1,\dots,d-1$.

Since $L_{d}$ is not the zero map, we can express one of the variables
comprising ${\bf m}$ in terms of the other $j-1$ variables and $l=L_{d}\left(\bfm\right)$.
Thus, the bound
\begin{align*}
Q\left(d,j,\mathcal{\mathbf{L}},N,\J,\epsilon\right) & \ll N^{\left(j-1\right)\left(1+\epsilon\right)+\epsilon/d}\sum_{1\leq\left|l\right|\ll_{{\bf L}}N^{1+\epsilon}}\left|l\right|^{-\frac{1}{d}}\sum_{t\leq N}t^{-\frac{A}{d}+1-\frac{1}{d}}\sum_{\substack{h_{i}\le t\\
i\le d-1
}
}\left(h_{1}\ldots h_{d-1}\right)^{-\frac{1}{d}}\\
 & \ll N^{j-\frac{1}{d}-\min\left(\frac{A}{d}-d,0\right)+\left(j+1\right)\epsilon}
\end{align*}
produces the required estimate.
\end{proof}
\begin{prop}
\label{prop: reduction expectation}Let $k\geq2$, and $\epsilon>0$.
If $0<d\leq k$, then 
\begin{equation}
S(\E_{k,d}^{\epsilon},\J)=\O(N^{k-1-\frac{1}{d}-\min\left(\frac{A}{d}-d,0\right)+\epsilon})\label{eq: expectation variable reduction}
\end{equation}
as $N\to\infty$. 
\end{prop}

\begin{proof}
For (possibly empty) index sets $\I_{1},\I_{2}\subseteq\left[k\right]$,
denote by $\E_{k,d}^{\epsilon}\left(\I_{1},\I_{2}\right)$ the set
of $\bfz=\left({\bf u,{\bf x}}\right)\in\E_{k,d}^{\epsilon}$ such
that 
\[
\left\{ i\in\left[k\right]:\,x_{i}=1\right\} =\I_{1}
\]
and
\[
\left\{ i\in\left[k\right]\setminus\I_{1}:\,u_{i}=0\right\} =\I_{2}.
\]
Assume that the set $\E_{k,d}^{\epsilon}\left(\I_{1},\I_{2}\right)$
is nonempty. Then $d=k-\#\left(\I_{1}\cup\mathcal{I}_{2}\right)$,
and since $x_{i}$ are distinct we have $\#\I_{1}\le1$.

Consider the sum
\begin{equation}
S\bigl(\E_{k,d}^{\epsilon}\left(\I_{1},\I_{2}\right),\J\bigr)\eqdef\sum_{\bfz\in\E_{k,d}^{\epsilon}\left(\I_{1},\I_{2}\right)}\left|I\left(\phi\left(\bfz,\cdot\right),\J\right)\right|\label{eq:Sum_Ek_II}
\end{equation}
and recall that the summation in (\ref{eq:Sum_Ek_II}) is over $\bfz=\left({\bf u,{\bf x}}\right)$
belonging to a subset of $\mathcal{U}_{k}^{\epsilon}\times\calBN$.
We first determine the constraints on ${\bf u}$. By the conditions
$u_{i}=0$ ($i\in\mathcal{I}_{2})$ and $u_{1}+\dots+u_{k}=0$, the
summation is restricted to $\bfu$ whose entries linearly depend on
$k-1-\#\mathcal{I}_{2}$ of the variables $u_{1},\dots,u_{k},$ i.e.,
there exists a set
\[
\left\{ j_{1},\dots,j_{k-1-\#\mathcal{I}_{2}}\right\} \subseteq\left[k\right]
\]
such that each $u_{i}$ is a linear combination of $u_{j_{1}},\dots,u_{j_{k-1-\#\mathcal{I}_{2}}}$
(since ${\bf u}\ne{\bf 0}_{k}$, we have $\#\mathcal{I}_{2}<k-1$).
For each $\bfz\in\E_{k,d}^{\epsilon}\left(\I_{1},\I_{2}\right)$,
we can then write 
\begin{equation}
\phi\left(\bfz,\alpha\right)=\sum_{i\in\left[k\right]\setminus\mathcal{I}_{2}}L_{i}(u_{j_{1}},\dots,u_{j_{k-1-\#\mathcal{I}_{2}}})x_{i}^{\alpha}\label{eq:nondeg_sum}
\end{equation}
where $L_{i}$ are nonzero linear combinations of the variables $u_{j_{1}},\dots,u_{j_{k-1-\#\mathcal{I}_{2}}}$
(determined only by $\mathcal{I}_{2}$).

There are $d$ non-constant terms in the sum (\ref{eq:nondeg_sum})
corresponding to the indices $i\in\left[k\right]\setminus\left(\mathcal{I}_{1}\cup\mathcal{I}_{2}\right)$
(note that if $\I_{1}$ is nonempty, then one of the terms in the
sum (\ref{eq:nondeg_sum}) is constant). Moreover, for each $i\in\I_{2}$,
the value of $x_{i}$ does not affect $\phi\left(\bfz,\alpha\right)$,
and $x_{i}$ ranges between $2$ and $N$, so that the function (\ref{eq:nondeg_sum})
appears $O(N^{\#\mathcal{I}_{2}})$ times in (\ref{eq:Sum_Ek_II})
upon summing over ${\bf z}$. Hence, letting $\mathbf{L}=\left(L_{i}\right)_{i\in\left[k\right]\setminus\left(\mathcal{I}_{1}\cup\mathcal{I}_{2}\right)}$,
we have
\[
S(\E_{k,d}^{\epsilon}\left(\I_{1},\I_{2}\right),\J)\ll N^{\#\I_{2}}Q\left(d,k-1-\#\mathcal{I}_{2},\mathbf{L},N,\J,\epsilon\right).
\]
By Lemma \ref{lem:Q_Bound}, we have
\[
N^{\#\I_{2}}Q\left(d,k-1-\#\mathcal{I}_{2},\mathbf{L},N,\J,\epsilon\right)\ll N^{k-1-\frac{1}{d}-\min\left(\frac{A}{d}-d,0\right)+\epsilon}.
\]
Summing over all configurations $\E_{k,d}^{\epsilon}\left(\I_{1},\I_{2}\right)$
completes the proof. 
\end{proof}
The argument for majorizing $S(\V_{k,d}^{\epsilon},\J)$ by a suitably
weighted sum $Q\bigl(d,j,{\bf L},N,\J,\epsilon\bigr)$ for some $d,j,{\bf L}$
is similar but the combinatorics is somewhat more technical. 
\begin{prop}
\label{prop: variance reduction}Let $k\geq2$, $\epsilon>0$. If
$d=0$ then 
\begin{equation}
S(\V_{k,0}^{\epsilon},\J)=\O(N^{2k-1+\epsilon})\label{eq: most degenerate case}
\end{equation}
as $N\to\infty$, and if $0<d\leq2k$ then
\begin{equation}
S(\V_{k,d}^{\epsilon},\J)=\O(N^{2k-2-\frac{1}{d}-\min\left(\frac{A}{d}-d,0\right)+\epsilon})\label{eq: variance variable reduction}
\end{equation}
as $N\to\infty$.
\end{prop}

\begin{proof}
Let $0\leq d\leq2k$, and let $\I_{1},\I_{1}',\I_{2},\I_{2}',\I_{3},\I_{3}',\I_{4},\I_{4}'\subseteq\left[k\right]$
be (possibly empty) sets of indices. Fixing $\bftau\eqdef\left(\I_{1},\I_{1}',\I_{2},\I_{2}',\I_{3},\I_{3}',\I_{4},\I_{4}'\right)$,
we denote by $\V_{k,d}^{\epsilon}\left(\bftau\right)$ the set of
vectors $\bfz=\left({\bf u},{\bf v},{\bf x},{\bf y}\right)\in\V_{k,d}^{\epsilon}$
for which 
\begin{align*}
\{i\in[k]:\,x_{i}=1\}=\I_{1},\quad\{j\in[k]:\,y_{j}=1\} & =\I_{1}',\\
\{i\in[k]\setminus\mathcal{I}_{1}:\,\exists_{j(i)\in[k]}\,x_{i}=y_{j(i)}\} & =\mathcal{I}_{2},\\
\{j\in[k]\setminus\mathcal{I}_{1}':\,\exists_{i(j)\in[k]}\,x_{i\left(j\right)}=y_{i}\} & =\mathcal{I}_{2}',\\
\{i\in\I_{2}:u_{i}+v_{j(i)}=0,\mathrm{where}\,j(i)\,\mathrm{is}\,\mathrm{s.t.}\,x_{i}=y_{j(i)}\} & =\I_{3},\\
\{j\in\I_{2}':u_{i\left(j\right)}+v_{j}=0,\mathrm{where}\,i(j)\,\mathrm{is}\,\mathrm{s.t.}\,x_{i\left(j\right)}=y_{j}\} & =\I_{3}',
\end{align*}
and 
\begin{align*}
\{i\in\left[k\right]\setminus(\mathcal{I}_{1}\cup\I_{2}):u_{i}=0\} & =\I_{4}\\
\{j\in\left[k\right]\setminus(\mathcal{I}_{1}'\cup\I_{2}'):v_{j}=0\} & =\I_{4}'.
\end{align*}
Assume that the set $\V_{k,d}^{\epsilon}\left(\bftau\right)$ is nonempty.
Then $\#\I_{2}=\#\I_{2}'$, $\#\I_{3}=\#\I_{3}'$ and
\begin{equation}
d=2k-\left(\#\I_{1}+\#\I_{1}'+\#\mathcal{I}_{2}+\#\I_{3}+\#\I_{4}+\#\I_{4}'\right).\label{eq:d_equation}
\end{equation}

Consider the sum
\begin{equation}
S\bigl(\V_{k,d}^{\epsilon}(\bftau),\J\bigr)\eqdef\sum_{\bfz\in\V_{k,d}^{\epsilon}\left(\bftau\right)}\left|I\left(\phi\left(\bfz,\cdot\right),\J\right)\right|.\label{eq:Sum_V_tau}
\end{equation}
We first consider the constraints on ${\bf u},{\bf v}$ when summing
in (\ref{eq:Sum_V_tau}) over $\bfz=\left({\bf u},{\bf v},{\bf x},{\bf y}\right)\in\V_{k,d}^{\epsilon}\left(\bftau\right)$:\\
\\
i) The conditions $u_{i}=0$ ($i\in\mathcal{I}_{4})$ and $u_{1}+\dots+u_{k}=0$
determine $\#\mathcal{I}_{4}+1$ of the variables $u_{i}$ in terms
of the other $k-1-\#\mathcal{I}_{4}$ variables $u_{i}$. Note that
${\bf u}\ne{\bf 0}_{k}$, so that $\#\mathcal{I}_{4}<k-1$. \\
\\
ii) The conditions $v_{j}=0$ ($j\in\I_{4}')$, $u_{i\left(j\right)}+v_{j}=0$,
$j\in\I_{3}'$, determine $\#\mathcal{I}_{4}'+\#\mathcal{I}_{3}'$
of the variables $v_{j}$ in terms of the variables $u_{i}.$\\
\\
iii) The condition $v_{1}+\dots+v_{k}=0$ trivializes if $\mathcal{I}_{3}\cup\I_{4}=\mathcal{I}_{3}'\cup\I_{4}'=\left[k\right]$.
Otherwise, if $\mathcal{I}_{3}'\cup\I_{4}'\ne\left[k\right]$, it
determines another variable $v_{j}$ ($j\notin\mathcal{I}_{3}'\cup\I_{4}')$
in terms of the rest of the variables. If $\mathcal{I}_{3}'\cup\I_{4}'$=$\left[k\right]$
and $\mathcal{I}_{3}\cup\I_{4}\ne\left[k\right]$, it determines another
variable $u_{i}$ in terms of the rest of the variables.\\

To conclude, we have found that there exist sets
\[
\left\{ i_{1},\dots,i_{l}\right\} \subseteq\left[k\right],\left\{ j_{1},\dots,j_{m}\right\} \subseteq\left[k\right]
\]
so that the variables $u_{1},\dots,u_{k},v_{1},\dots,v_{k}$ linearly
depend on $u_{i_{1}},\dots u_{i_{l}},v_{j_{1}},\dots v_{j_{m}}$,
and
\begin{equation}
l+m=\left(k-1-\#\mathcal{I}_{4}\right)+\left(k-\#\mathcal{I}_{3}'-\#\mathcal{I}_{4}'\right)-1=2k-2-\#\mathcal{I}_{3}'-\#\mathcal{I}_{4}-\#\mathcal{I}_{4}'\label{eq:l+m_non_deg}
\end{equation}
unless $\mathcal{I}_{3}\cup\I_{4}=\mathcal{I}_{3}'\cup\I_{4}'=\left[k\right]$,
in which case we have
\begin{equation}
l+m=\left(k-1-\#\mathcal{I}_{4}\right)+\left(k-\#\mathcal{I}_{3}'-\#\mathcal{I}_{4}'\right)=2k-1-\#\mathcal{I}_{3}'-\#\mathcal{I}_{4}-\#\mathcal{I}_{4}'.\label{eq:l+m_deg}
\end{equation}
For each ${\bf z}\in\V_{k,d}^{\epsilon}\left(\bftau\right)$, we can
then write
\begin{align}
\phi\left(\bfz,\alpha\right) & =\sum_{i\in\left[k\right]\setminus\left(\mathcal{I}_{3}\cup\mathcal{I}_{4}\right)}L_{i}\left(u_{i_{1}},\dots u_{i_{l}},v_{j_{1}},\dots v_{j_{m}}\right)x_{i}^{\alpha}\label{eq:non_deg_sum_var}\\
 & +\sum_{j\in\left[k\right]\setminus\left(\mathcal{I}_{2}'\cup\mathcal{I}_{4}'\right)}L_{j}\left(u_{i_{1}},\dots u_{i_{l}},v_{j_{1}},\dots v_{j_{m}}\right)y_{i}^{\alpha}\nonumber 
\end{align}
where $L_{i}$ are nonzero linear combinations of the variables $u_{i_{1}},\dots u_{i_{l}},v_{j_{1}},\dots v_{j_{m}}$
(determined by $\bftau$).

The total number of non-constant terms in (\ref{eq:non_deg_sum_var})
is $d$ (see (\ref{eq:d_equation}); note that if at least one of
the sets $\mathcal{I}_{1}$ or $\mathcal{I}_{1}'$ is nonempty, then
one or two terms in the sum (\ref{eq:nondeg_sum}) are constant).
For each $i\in\mathcal{I}_{3}\cup\mathcal{I}_{4}$, the value of $x_{i}$
does not affect $\phi\left(\bfz,\alpha\right)$, and $x_{i}$ ranges
between $2$ and $N$. Likewise, for each $j\in\mathcal{I}_{4}'$,
the value of $y_{j}$ does not affect $\phi\left(\bfz,\alpha\right)$,
and $y_{j}$ ranges between $2$ and $N$. Hence, the function (\ref{eq:non_deg_sum_var})
appears $O(N^{\#\I_{3}+\#\I_{4}+\#\I_{4}'})$ times in (\ref{eq:Sum_V_tau})
upon summing over ${\bf z}$.

If $d=0,$ then the phase function is constant (in fact, it must vanish),
and therefore 
\[
S\bigl(\V_{k,0}^{\epsilon}(\bftau),\J\bigr)\ll N^{\#\I_{3}+\#\I_{4}+\#\I_{4}'}N^{\left(1+\epsilon\right)\left(l+m\right)}\ll N^{\left(2k-1\right)\left(1+\epsilon\right)}
\]
in either of the cases (\ref{eq:l+m_non_deg}), (\ref{eq:l+m_deg}).
If $d>0$, we can choose 
\[
\mathbf{L}=\left(L_{i},L_{j}\right)_{i\in\left[k\right]\setminus\left(\mathcal{I}_{1}\cup\mathcal{I}_{3}\cup\mathcal{I}_{4}\right),\,j\in\left[k\right]\setminus\left(\mathcal{I}_{1}'\cup\mathcal{I}_{2}'\cup\mathcal{I}_{4}'\right)}
\]
and get that
\[
S\bigl(\V_{k,d}^{\epsilon}(\bftau),\J\bigr)\ll N^{\#\I_{3}+\#\I_{4}+\#\I_{4}'}Q\left(d,l+m,\mathbf{L},N,\J,\epsilon\right).
\]
Clearly, either $\mathcal{I}_{3}\cup\I_{4}\ne\left[k\right]$ or $\mathcal{I}_{3}'\cup\I_{4}'\ne\left[k\right]$,
so that by (\ref{eq:l+m_non_deg}) we have
\[
Q\left(d,l+m,\mathbf{L},N,\J,\epsilon\right)=Q\left(d,2k-2-\#\mathcal{I}_{3}-\#\mathcal{I}_{4}-\#\mathcal{I}_{4}',\mathbf{L},N,\J,\epsilon\right).
\]
By Lemma \ref{lem:Q_Bound}, we establish the bound
\begin{align*}
N^{\#\I_{3}+\#\I_{4}+\#\I_{4}'}Q\left(d,2k-2-\#\mathcal{I}_{3}-\#\mathcal{I}_{4}-\#\mathcal{I}_{4}',\mathbf{L},N,\J,\epsilon\right)\ll N^{2k-2-\frac{1}{d}-\min\left(\frac{A}{d}-d,0\right)+\epsilon}.
\end{align*}
As there are $\ll_{k}1$ many sets $\V_{k,d}^{\epsilon}\left(\bftau\right)$,
summing over $\bftau$ concludes the proof. 
\end{proof}
\begin{cor}
\label{cor:MainCorollary}For each $A>k^{2}-k-1$ there exists $\rho=\rho\left(A\right)>0$
such that for any $0<d\leq k$, 
\begin{equation}
S(\E_{k,d}^{\epsilon},\J)=\O(N^{k-\rho})\label{eq: expectation degenerate}
\end{equation}
as $N\to\infty$. Further, for each $A>4k^{2}-4k-1$ there exists
$\rho=\rho\left(A\right)>0$ such that for any $0\leq d\leq2k$,
\begin{equation}
S(\V_{k,d}^{\epsilon},\J)=\O(N^{2k-\rho})\label{eq: variance degenerate}
\end{equation}
as $N\to\infty$.
\end{cor}

\begin{proof}
Recall that by Proposition \ref{prop: reduction expectation} we have
\[
S(\E_{k,d}^{\epsilon},\J)=\O\bigl(N^{k-1-\frac{1}{d}-\min\left(\frac{A}{d}-d,0\right)+\epsilon}\bigr).
\]
 Hence, we obtain (\ref{eq: expectation degenerate}) when $-1-\frac{1}{d}-\frac{A}{d}+d<0$,
or equivalently when $A>d(d-1)-1$ for all $0<d\le k$. Since $d\mapsto d(d-1)-1$
is increasing on the interval $\left[1,k\right]$ it attains its maximum
value at $d=k$, which is equal to $k^{2}-k-1$. Hence, (\ref{eq: expectation degenerate})
holds whenever $A>k^{2}-k-1$.

Now recall that Proposition \ref{prop: variance reduction} yields
$S(\V_{k,0}^{\epsilon},\J)=\O(N^{2k-1+\epsilon})$ for $d=0$ and
\begin{align*}
S(\V_{k,d}^{\epsilon},\J) & =\O(N^{2k-2-\frac{1}{d}-\min\left(\frac{A}{d}-d,0\right)+\epsilon})
\end{align*}
for $0<d\le2k$. Thus, (\ref{eq: variance degenerate}) holds when
$-2-\frac{1}{d}-\frac{A}{d}+d<0$, i.e., when $A>d(d-2)-1$ for all
$0<d\le2k$. Since $d(d-2)-1$ is increasing as a function of $d\in\left[1,2k\right]$,
the maximum is attained at $d=2k$ and is equal to $4k^{2}-4k-1$.
Therefore, (\ref{eq: variance degenerate}) holds whenever $A>4k^{2}-4k-1$.
\end{proof}
Substituting the bound (\ref{eq: expectation degenerate}) in (\ref{eq: rep of expectation remainder}),
we can now find a regime in which expectation of $R_{k}\left(f,\cdot,N\right)$
is asymptotic to the average of $f$. We formulate the next proposition
for $k>2$, since for $k=2$ Proposition \ref{prop: Expectation}
yielded a stronger result (holding for $A>0$).
\begin{prop}
\label{prop:Expectation-k>2}Let $k>2$, $A>k^{2}-k-1$ and $\J$
be given by (\ref{def: J}). Then there exists $\rho=\rho\left(A\right)>0$
such that
\[
\int_{\mathcal{J}}R_{k}\left(f,\alpha,N\right)\,\mathrm{d}\alpha=\int_{\mathbb{R}^{k-1}}f\left(\bfx\right)\,\mathrm{d}\bfx+\O(N^{-\rho})
\]
as $N\to\infty$.
\end{prop}

Finally, Theorem \ref{thm: variance bound} also follows from Corollary
\ref{cor:MainCorollary}:
\begin{proof}[Proof of Theorem \ref{thm: variance bound}]
The bound (\ref{eq: variance bound}), $k\ge2$, follows by substituting
(\ref{eq: variance degenerate}) in (\ref{eq: rep of variance remainder}).
\end{proof}

\section{Proof of Theorem \ref{thm: higher order} and Corollary \ref{cor: approximation to gap distribution}}

With the variance bound from Theorem \ref{thm: variance bound} at
hand, we can deduce Theorem \ref{thm: higher order} by rather soft
arguments from a general principle. Although the argument is fairly
standard, we have not found it stated explicitly in the literature
in a form that readily applies to our case, so we decided to give
the details in full. The following proposition deduces Theorem \ref{thm: higher order}
from the variance bound, recorded in Theorem \ref{thm: variance bound},
at once.
\begin{prop}
\label{prop: L2 implies almost sure convergence}Let $k\ge2,$ let
$\mathcal{I}\subset\mathbb{R}$ be a bounded interval, and let $c_{k}\left(N\right)$
be a sequence satisfying $c_{k}\left(N\right)\to1$ as $N\to\infty$.
Suppose we are given a real-valued sequence $\left(\vartheta_{n}(\alpha)\right)_{n\geq1}$
for each $\alpha\in\mathcal{I}$ so that $\mathcal{I}\ni\alpha\mapsto\vartheta_{n}(\alpha)$
is a continuous map for each fixed $n\geq1$. Assume that there exists
$\rho>0$ such that for all $f\in C_{c}^{\infty}(\mathbb{R}^{k-1})$,
\begin{align}
\int_{\mathcal{I}}\left(R_{k}(f,(\vartheta_{n}(\alpha)),N)-c_{k}\left(N\right)\int_{\mathbb{R}^{k-1}}f(\bfx)\,\mathrm{d}\bfx\right)^{2}\,\mathrm{d}\alpha & =\O(N^{-\rho}).\label{eq:Var_Assum}
\end{align}
as $N\to\infty$, then the sequence $(\vartheta_{n}(\alpha))_{n\geq1}$
has Poissonian $k$-level correlation for almost every $\alpha\in\mathcal{I}$.
\end{prop}

First we record a useful lemma that allows us to pass from the convergence
of a sub-sequence to the convergence of the entire sequence (extending
\cite[Lem. 3.1]{Rudnick-Technau} which was established for $k=2$).
\begin{lem}
\label{lem:N_m_to_N}Let $(\vartheta_{n})_{n\geq1}$ be a real-valued
sequence. If there is an increasing sequence $\left(N_{m}\right)_{m\geq1}$
of positive integers so that
\begin{equation}
\lim_{m\rightarrow\infty}\frac{N_{m+1}}{N_{m}}=1\label{eq:N_m_ratio}
\end{equation}
and so that
\begin{equation}
\lim_{m\rightarrow\infty}R_{k}\left(f,(\vartheta_{n}),N_{m}\right)=\int_{\mathbb{R}^{k-1}}\,f(\bfx)\,\mathrm{d}\bfx\label{eq: convergence on subsequence}
\end{equation}
holds for all $f\in C_{c}^{\infty}(\mathbb{R}^{k-1})$, then 
\begin{equation}
\lim_{N\rightarrow\infty}R_{k}\left(f,(\vartheta_{n}),N\right)=\int_{\mathbb{R}^{k-1}}\,f(\bfx)\,\mathrm{d}\bfx\label{eq: convergence}
\end{equation}
holds for all $f\in C_{c}^{\infty}(\mathbb{R}^{k-1})$. Moreover,
(\ref{eq: convergence}) holds for all indicator functions $f=1_{\Pi}$
of boxes $\Pi\subseteq\mathbb{R}^{k-1}$.
\end{lem}

\begin{proof}
First we argue that if the assumption (\ref{eq: convergence on subsequence})
is true for all $f\in C_{c}^{\infty}(\mathbb{R}^{k-1})$ then it also
holds for all indicator functions $1_{\Pi}$ of boxes $\Pi=\left[a_{1},b_{1}\right]\times\ldots\times\left[a_{k-1},b_{k-1}\right]$.
Fix $\delta>0$ and choose $f_{-},f_{+}\in C_{c}^{\infty}(\mathbb{R}^{k-1})$
such that $f_{-}\leq1_{\Pi}\leq f_{+}$ and 
\[
\int_{\mathbb{R}^{k-1}}\left(f_{+}\left(\bfx\right)-f_{-}\left(\bfx\right)\right)\,\mathrm{d}\bfx<\delta.
\]
Then by the definition of the correlation sum (\ref{eq:R_k_def})
we have 
\[
R_{k}\left(f_{-},(\vartheta_{n}),N\right)\leq R_{k}\left(1_{\Pi},(\vartheta_{n}),N\right)\leq R_{k}\left(f_{+},(\vartheta_{n}),N\right).
\]
Thus
\[
\limsup_{m\rightarrow\infty}R_{k}\left(1_{\Pi},(\vartheta_{n}),N_{m}\right)\leq\limsup_{m\rightarrow\infty}R_{k}\left(f_{+},(\vartheta_{n}),N_{m}\right)=\int_{\mathbb{R}^{k-1}}f_{+}\left(\bfx\right)\,\mathrm{d}\bfx
\]
and 
\[
\liminf_{m\rightarrow\infty}R_{k}\left(1_{\Pi},(\vartheta_{n}),N_{m}\right)\geq\liminf_{m\rightarrow\infty}R_{k}\left(f_{-},(\vartheta_{n}),N_{m}\right)=\int_{\mathbb{R}^{k-1}}f_{-}\left(\bfx\right)\,\mathrm{d}\bfx.
\]
Therefore, 
\begin{align*}
0 & \le\limsup_{m\rightarrow\infty}R_{k}\left(1_{\Pi},(\vartheta_{n}),N_{m}\right)-\liminf_{m\rightarrow\infty}R_{k}\left(1_{\Pi},(\vartheta_{n}),N_{m}\right)\\
 & \le\int_{\mathbb{R}^{k-1}}\left(f_{+}\left(\bfx\right)-f_{-}\left(\bfx\right)\right)\,\mathrm{d}\bfx<\delta.
\end{align*}
Since $\delta>0$ was arbitrary, we conclude that 
\begin{equation}
\lim_{m\rightarrow\infty}R_{k}\left(1_{\Pi},(\vartheta_{n}),N_{m}\right)=\int_{\mathbb{R}^{k-1}}1_{\Pi}\left(\bfx\right)\,\mathrm{d}\bfx\label{eq:LimitIndicators}
\end{equation}
which verifies (\ref{eq: convergence on subsequence}) for $f=1_{\Pi}$. 

Given a positive integer $N$, we can find $m\geq1$ such that $N_{m}\leq N<N_{m+1}$.
Moreover,
\[
R_{k}\left(1_{\Pi},(\vartheta_{n}),N\right)=\frac{1}{N}\#\left\{ \bfx\in\calBN:\,\vartheta_{x_{i}}-\vartheta_{x_{i+1}}\in\left(\frac{a_{i}}{N},\frac{b_{i}}{N}\right)+\mathbb{Z},\hspace{1em}i=1,\ldots,k-1\right\} .
\]
The limit (\ref{eq:N_m_ratio}) implies that when $N$ is sufficiently
large we have $\frac{N_{m}}{N}=1+o\left(1\right)$. Given $\delta>0,$
we therefore let
\[
\Pi'=\left[a_{1}-\delta,b_{1}+\delta\right]\times\ldots\times\left[a_{k-1}-\delta,b_{k-1}+\delta\right]
\]
and see that for sufficiently large $N$ we have
\begin{align*}
 & R_{k}\left(1_{\Pi},(\vartheta_{n}),N\right)\\
 & \leq\frac{1}{N_{m}}\#\left\{ \bfx\in\calBN:\,\vartheta_{x_{i}}-\vartheta_{x_{i+1}}\in\left(\frac{a_{i}\cdot\frac{N_{m}}{N}}{N_{m}},\frac{b_{i}\cdot\frac{N_{m}}{N}}{N_{m}}\right)+\mathbb{Z}\hspace{1em},i=1,\ldots,k-1\right\} \\
 & \leq\frac{1}{N_{m}}\#\left\{ \bfx\in\calBN:\,\vartheta_{x_{i}}-\vartheta_{x_{i+1}}\in\left(\frac{a_{i}-\delta}{N_{m}},\frac{b_{i}+\delta}{N_{m}}\right)+\mathbb{Z}\hspace{1em},i=1,\ldots,k-1\right\} 
\end{align*}
where the right hand side is $R_{k}\left(1_{\Pi'},(\vartheta_{n}),N_{m}\right)$.
Thus, we conclude that 
\begin{align*}
\limsup_{N\rightarrow\infty}R_{k}\left(1_{\Pi},(\vartheta_{n}),N\right) & \leq\limsup_{m\rightarrow\infty}R_{k}\left(1_{\Pi'},(\vartheta_{n}),N_{m}\right)=\int_{\mathbb{R}^{k-1}}1_{\Pi'}\left(\bfx\right)\,\mathrm{d}\bfx.
\end{align*}
Recalling the definition of $\Pi'$, we clearly have 
\[
\int_{\mathbb{R}^{k-1}}1_{\Pi'}\left(\bfx\right)\,\mathrm{d}\bfx=\int_{\mathbb{R}^{k-1}}1_{\Pi}\left(\bfx\right)\,\mathrm{d}\bfx+\O\left(\delta\right).
\]
Since $\delta$ was arbitrary, we infer that
\[
\limsup_{N\rightarrow\infty}R_{k}\left(1_{\Pi},(\vartheta_{n}),N\right)\leq\int_{\mathbb{R}^{k-1}}1_{\Pi}\left(\bfx\right)\,\mathrm{d}\bfx
\]
and a similar argument shows that 
\[
\liminf_{N\rightarrow\infty}R_{k}\left(1_{\Pi},(\vartheta_{n}),N\right)\geq\int_{\mathbb{R}^{k-1}}1_{\Pi}\left(\bfx\right)\,\mathrm{d}\bfx.
\]
This establishes (\ref{eq: convergence}) for all functions $f=1_{\Pi}.$
Since every $f\in C_{c}^{\infty}(\mathbb{R}^{k-1})$ can be approximated
from below and from above by a linear combination of indicator functions
of boxes, (\ref{eq: convergence}) holds for smooth compactly supported
functions as well (by the same argument we detailed above to prove
(\ref{eq:LimitIndicators})).
\end{proof}
We are now ready to prove Proposition \ref{prop: L2 implies almost sure convergence}.
\begin{proof}[Proof of Proposition \ref{prop: L2 implies almost sure convergence}]
For each $m\geq1$, let $N_{m}=\lfloor m^{2/\rho}\rfloor$. For each
fixed $f\in C_{c}^{\infty}(\mathbb{R}^{k-1})$ define 
\[
X_{m}\left(\alpha\right)=\left|R_{k}\left(f,(\vartheta_{n}(\alpha)),N_{m}\right)-c_{k}\left(N_{m}\right)\int_{\mathbb{R}^{k-1}}f\left(x\right)\,\mathrm{d}x\right|^{2}.
\]
By (\ref{eq:Var_Assum}), the $L^{1}$-norms of $X_{m}\geq0$ on $\mathcal{I}$
are summable. Changing the order of summation and integration yields
\[
\int_{\mathcal{I}}\sum_{m\geq1}X_{m}\left(\alpha\right)\text{d}\alpha<\infty,
\]
and therefore for almost all $\alpha\in\mathcal{I}$ we have
\[
\sum_{m\geq1}X_{m}\left(\alpha\right)<\infty.
\]
In particular $X_{m}\left(\alpha\right)\rightarrow0$ for almost all
$\alpha\in\mathcal{I}$, and hence (\ref{eq: convergence on subsequence})
is satisfied for almost all $\alpha\in\mathcal{I}$ for our fixed
$f$; by a standard diagonal argument (approximating from above and
below by functions $f_{i}$ belonging to a countable dense set in
$C_{c}^{\infty}(\mathbb{R}^{k-1}$)), we conclude that for almost
all $\alpha\in\mathcal{I}$, (\ref{eq: convergence on subsequence})
holds for all $f\in C_{c}^{\infty}(\mathbb{R}^{k-1})$. Hence by Lemma
\ref{lem:N_m_to_N}, the limit (\ref{eq: convergence}) holds for
almost every $\alpha\in\mathcal{I}$, completing the proof.
\end{proof}
To prove Corollary \ref{cor: approximation to gap distribution},
we require a well-known relation between the gap distribution and
the correlation functions. Let
\[
\Delta_{k-1}=\Bigl\{(x_{1},\ldots,x_{k-1})\in\mathbb{R}_{>0}^{k-1}:\,\sum_{1\le i\leq k-1}x_{i}<1\Bigr\}
\]
denote the standard open $k-1$-simplex. For $x>0$, let $1_{x\Delta_{k-1}}$
be the indicator function of the dilation $x\Delta_{k-1}$.
\begin{lem}
\label{lem: Par--Zeev}Let $(\vartheta_{n})_{n\geq1}$ be a real-valued
sequence, and let $K\geq1$. For all $x>0$, we have
\[
\sum_{2\le k\leq2K+1}(-1)^{k}R_{k}\left(1_{x\Delta_{k-1}},(\vartheta_{n}),N\right)\le g\left(x,(\vartheta_{n}),N\right)\le\sum_{2\le k\leq2K}(-1)^{k}R_{k}\left(1_{x\Delta_{k-1}},(\vartheta_{n}),N\right).
\]
\end{lem}

\begin{proof}
The claim follows from Lemma 11 and (A.2) of \cite{Kurlberg-Rudnick}.
\end{proof}
We are now in the position to prove Corollary \ref{cor: approximation to gap distribution}.
\begin{proof}[Proof of Corollary \ref{cor: approximation to gap distribution}]
By Theorem \ref{thm: higher order}, for almost all
\[
\alpha>4(2K+1)^{2}-4(2K+1)-1=16K^{2}+8K-1,
\]
the $k$-level correlation functions $R_{k}$ are Poissonian for all
$2\le k\leq2K+1$, so that as $N\to\infty$, $R_{k}\left(1_{x\Delta_{k-1}},(\vartheta_{n}),N\right)$
converges to the volume of $x\Delta_{k-1}$ which is equal to $x^{k-1}/(k-1)!$.
The claimed inequalities now follow from Lemma \ref{lem: Par--Zeev}.
\end{proof}

\end{document}